\documentclass[journal]{IEEEtran}
\usepackage{amsmath,graphicx,url}
\usepackage{amssymb}
\usepackage{graphicx}
\usepackage{epstopdf}
\usepackage{tabularx}
\usepackage{color}
\usepackage{times}
\usepackage{epsfig}
\usepackage{amsmath}
\usepackage{array}
\usepackage{threeparttable}
\usepackage{algorithm}
\usepackage{color}
\usepackage{lineno}
\usepackage{bm}
\usepackage{subfigure}
\usepackage{cases}
\usepackage{epsfig,amssymb,subfigure,amsmath,version,amsthm}
\usepackage{amssymb,version,graphicx,fancybox,mathrsfs}
\usepackage{graphicx,mathrsfs,fancyhdr}
\usepackage{subfigure,slashbox,rotating,version,fancybox,pifont,booktabs,epsfig}
\usepackage{setspace}

\newtheorem{thm}{\bf Theorem}

\newenvironment{theorem}{\begin{thm}} {\end{thm}}
\newtheorem{cor}{\bf Corollary}

\newtheorem{lmm}{\bf Lemma}

\newenvironment{lemma}{\begin{lmm}}{\end{lmm}}

\def \epsilon {{\varepsilon}}

\newtheorem{remark}{Remark}[section]
\ifCLASSINFOpdf
\else
\fi
\begin{document}
%
\title{$L^0$-regularized Variational Methods for Sparse Phase Retrieval}

\author{
Yuping Duan, Chunlin Wu,
Zhi-Feng Pang,
Huibin Chang${}^*$
\thanks{Y. Duan is with Center for Applied Mathematics, Tianjin University, Tianjin, China.}
\thanks{C. Wu is with School of Mathematical Sciences, Nankai University, Tianjin, China.}
\thanks{Z.-F. Pang is with Department of Mathematics and Statistics, Henan University, Kaifeng, China
}
\thanks{${}^*$Corresponding author. H. Chang is  with Department of Mathematical Sciences, Tianjin Normal University, Tianjin, China, e-mail:changhuibin@gmail.com.
}
%

}

%

%
%

\markboth{Journal of \LaTeX\ Class Files,~Vol.~14, No.~8, August~2015}%
{Shell \MakeLowercase{\textit{et al.}}: Bare Demo of IEEEtran.cls for IEEE Journals}
%



\maketitle

\begin{abstract}
We study the problem of recovering the underlining sparse signals from clean or noisy phaseless measurements. Due to the  sparse \emph{prior} of signals, we adopt an $L^0$ regularized variational model to ensure only a small number of nonzero elements being recovered in the signal and two different formulations are established  in the modeling based on the choices of data fidelity, \emph{i.e.}, $L^2$ and $L^1$ norms. We also propose efficient algorithms based on the Alternating Direction Method of Multipliers (ADMM)  with convergence guarantee and nearly optimal computational complexity. Thanks to the existence of closed-form solutions to all subproblems, the proposed algorithm is very efficient with low computational cost in each iteration. Numerous experiments show that our proposed methods can recover sparse signals from phaseless measurements with higher successful recovery rates and lower computation cost compared with the state-of-art methods.
\end{abstract}

\begin{IEEEkeywords}
 Phase retrieval, $L^0$ regularization, Alternating Direction Method of Multipliers (ADMM), Dynamic Step,  Sparse Signals
\end{IEEEkeywords}

%
\IEEEpeerreviewmaketitle

\section{Introduction}\label{intro}
\IEEEPARstart{P}{hase} retrieval (PR) is known as the recovery of an underling signal from the magnitude of its Fourier transform \cite{patterson1934fourier,patterson1944ambiguities}. It has a wide applications in many areas of engineering and applied physics, including optical imaging \cite{walther1963question}, X-ray crystallography \cite{millane1990phase}, astronomical imaging \cite{fienup1987phase}, \emph{etc.}

Let $\bm x \in \mathbb{C}^{N}$ be a one-dimensional sampled unknown object. We mathematically set up the phase retrieval problem in a discretized environment as
\begin{equation}
\label{phase_retrieval}
\begin{split}
\mbox{\qquad To find} \quad &\bm x \\
\mbox{s.t.,}\quad~ & |\mathcal{F}\bm x|=b,
\end{split}
\end{equation}
where $b\in\mathbb R_+^N$ is a sampled measurement, $\mathcal{F}$ is the Discrete Fourier Transform (DFT) and $|\cdot|$ denotes the componentwise absolute value.
More generally, we can consider PR with an arbitrary linear operator $\mathcal A:\mathbb R^N\rightarrow\mathbb R^{\hat N}$ rather than DFT form of the phaseless measurements, where the data $b\in \mathbb R_+^{\hat N}$  are collected \footnote{For simplicity we still use the notation $b$ to represent the collected measurements} as below
  \[
  b=|\mathcal A \bm x|
  .\]


\subsection{Related Literatures}
Generally speaking, PR is  an ill-posed and challengeable quadratic  optimization problem, since there are possibly many nontrivial solutions  \cite{Sanz1985} and solving  a general quadratic problem is NP hard \cite{pardalos1991quadratic}. From algorithmic perspective, the alternating projections was pioneered by Gerchberg and Saxton \cite{gerchberg1972practical} to reconstruct
phase from two intensity measurements. Such kind of algorithms by alternating projections  has been further studied for PR with additional support set and positivity  constraint in \cite{fienup1982phase,elser2003phase,marchesini2013augmented,Luke2005,Marchesini2007}. More advanced alternating projection algorithms can be found in the review paper by Marchesini \cite{marchesini2007invited}. However, since the projections are onto non-convex sets, the iterative algorithms often fall into a local minimum \cite{bauschke2002phase}. Gradient descent schemes with adaptive steps were exploited for general PR \cite{candes2015phaseWF,chen2015solving}. Recently, insensitive studies focused on the algorithms with theoretically global convergence for the general PR. For instance, a global convergence for  Gaussian measurements was  analyzed by Netrapalli, Jain and  Sanghavi \cite{netrapalli2015phase}.  The convergence under a \emph{generic} frame  was proved by Marchesini \emph{et al.}  \cite{marchesini2015alternating}, and phase-synchronization  was used to accelerate the computation speed for large scale ptychographic PR. A second order trust-region algorithm was proposed by Sun, Qu, and Wright \cite{sun2016geometric} based on the geometric analysis of PR, which returned a solution sufficient close to the ground truth by collecting about $O(N\log^3(N))$ \emph{generic} measurements.  Another branch of methods convexify the PR by semi-definite programming (SDP) based PhaseLift \cite{candes2015phaselift} and PhaseCut \cite{Waldspurger2012}. PhaseMax \cite{goldstein2016phasemax,bahmani2016phase} was developed with much less computational cost, which  operates in the original signal dimension rather than in a higher dimension as SDP based methods.

Inspired by the success of compressed sensing \cite{donoho2006compressed,candes2006robust}, the sparsity \emph{prior} was added to PR problem  in order to increase the recovery possibilities and the quality of reconstructed signals, especially when the measurement is noisy and incomplete. Marchesini \emph{et al.} established a \textsc{shrink-wrap} algorithm  in \cite{marchesini:2003} by generalizing the regression model, where  the support of unknown signals are  slowly shrinked as iteration goes on. Similar methods based on soft thresholdings were studied by  Moravec, Romberg and Baraniuk \cite{moravec2007compressive}.  The SDP-based methods have also been applied to solve the sparse PR problems \cite{ohlsson2012cprl,li2013sparse}. 
Direct extension of the Fieup's method \cite{Fienup1982} was proposed with an additional sparsity constraint of the $L^0$ pseudo norm \cite{mukherjee2014fienup}. The $L^1$ regularization based variational method \cite{yang2013robust} was applied to the classical PR problem. Shechtman, Beck and  Eldar presented an efficient local search method for recovering a sparse signal for the classical PR \cite{shechtman2014gespar}. A  probabilistic method  based on the generalized approximate message passing algorithm was proposed and studied in \cite{schniter2015compressive}.   The shearlet and total variation  regularization methods  in \cite{loock2014phase,chang2016phase,chang2016} assumed the objects possess a sparse representation in the transform domain. Dictionary learning methods were proposed with the dictionary automatically learned from the redundant image patches \cite{tillmann2016dolphin,qiu2016undersampled}.  A general framework was proposed \cite{chang2016general} by combining the popular image filters as BM3D to denoise phaseless measurements.

On the other hand, people also studied the minimal measurements to recover $s-$sparse\footnote{Size of support set or number of nonzeros is $s$} signals. Ohlsson and Eldar \cite{ohlsson2014conditions} showed that at least $s^2-s+1$ Fourier measurements are sufficient for such signals.   Eldar and Mendelson \cite{eldar2014phase} proved that  for the $s-$sparse real-valued signals, $O(s\log(N/s)$ measurements are sufficient
for stable uniqueness. Xu and Wang \cite{wang2014phase} studied  the null
space property for $s-$sparse signals and showed that at least $2s$ and $4s-2$ measurements are sufficient assuming the rows of $\mathcal A$ being \emph{generically} chosen from $\mathbb R^N$ and $\mathbb C^N$ for the real and complex valued signals respectively. More related studies can be found in \cite{Shechtman2014} and references therein.



\subsection{Our Contribution}
In this paper, our discussion is based on the assumption that the signal is sparse in the original domain and may be contaminated by noises. In order to find the solution $\bm x$ stated in \eqref{phase_retrieval}, we set up a better-conditioned problem based on the regularization. More specifically, we introduce the $L^0$ pseudo norm as the regularization for $\bm x$ due to its sparsity \emph{prior} and propose the following minimization problem
\begin{equation}
\label{pr_constrained}
\begin{split}
\min_{\bm x} &~~\|\bm x \|_0\\
\mbox{s.t.,} &~~|\mathcal{A} \bm x|\approx b,
\end{split}
\end{equation}
where $$\|\bm x\|_0:=\#\left\{k\big| |\Re(\bm x)_i|+|\Im(\bm x)_i|\not=0, ~~i = 1, \cdots, N\right\}.$$
The constraint in \eqref{pr_constrained} is realized by minimizing the distance between the observed magnitude $b$ and the value $|\mathcal{A} \bm x|$, which is measured by $L^p$ distance, $p=1 \mbox{~and~} 2$. We design algorithms for the established non-convex optimization problem \eqref{pr_constrained} based on the operator splitting technique,  and  an alternating direction method of multipliers (ADMM) \cite{Wu&Tai2010} with dynamic steps is adopted for detailed implementation. Since each subproblem has the closed-form solution, the proposed algorithms are very efficient with low computational costs.
Note that our model does not need to know the support information of the signal in advance and is robust when $b$ is contaminated by noises as \cite{mukherjee2014fienup}. Numerous experiments demonstrate that  our proposed methods are capable of recovering signals with higher probability for noiseless measurements  and higher SNRs for noisy measurements compared with the state-of-art algorithms.
\subsection{Outline}

The rest of this paper is organized as follows. In Section \ref{sec2},  the $L^0$ regularized models are established, where the existence of the minimizers to the proposed models is obtained as well. Section \ref{sec3} discusses the fast numerical algorithms for solving the proposed models, and local convergence is provided as well. Experiments are performed in Section \ref{sec4} to  demonstrate the effectiveness and robustness of the proposed methods from clean and noisy  phaseless measurements for sparse signals. Conclusions and future works are given in Section \ref{sec6}.


\section{$L^0$-regularized Sparse PR Model}\label{sec2}
We first give a brief review of two state-of-art algorithms, \emph{i.e.,} Sparse Fieup algorithm and GESPEAR (a greedy sparse phase retrieval algorithm), and one can refer to \cite{Shechtman2014} for more related methods for sparse PR problems.
\subsection{Review of Sparse Fieup and GESPEAR}
\subsubsection{Sparse Fieup algorithm}
Mukherjee and  Seelamantula \cite{mukherjee2014fienup} proposed to solve a sparse PR problem (SPR), which is formulated as
\begin{equation}\label{SPR}
\text{To find~} \bm x, ~\mbox{s.t.},~ \bm x\in \mathcal C_s\cap\mathcal M,
\end{equation}
where
\[\mathcal C_s:=\{\bm x: \|\bm x\|_0\leq s\},\] and
\[
\mathcal M:=\{\bm x: |\mathcal F \bm x|=b\},
\]
with  the known $s$  nonzeros.
An alterative projection algorithm  was proposed as
\[
\bm x^{k+1}=\mathcal P_{\mathcal C_s} (\mathcal P_{\mathcal M}(\bm x^k))
\]
for the $k^{th}$ iteration, where  the projection operator is denoted as
 \[
\mathcal P_{\mathcal A}(\bm x):=\arg\min_{\bm y\in \mathcal A} \|\bm y-\bm x\|^2.
\]

\begin{remark}
In \cite{mukherjee2014fienup}, they proposed a more general models, where the set
 \[\hat C_s:=\{\bm x: x=\bm \Upsilon y, \|y\|_0\leq s\}\]
 such that the signal is sparse by  representing by some linear basis $\bm\Upsilon$. In order to solve the projection $\mathcal P_{\hat {\mathcal C}_s}$, the OMP algorithm was used. In this paper, we only consider the sparsity in the signal domain. Obviously for noisy data, one can not require the underling signal belongs to the set $\mathcal M.$
\end{remark}

\subsubsection{GESPAR}
 Shechtman, Beck and Eldar \cite{shechtman2014gespar} proposed to solve the following optimization problem
 \[
 \begin{split}
 \min\limits_{\bm x\in \mathbb R^n}&\quad\||\mathcal F \bm x|^2-b^2\|,\\
 \mbox{s.t.}, &\quad \|\bm x\|_0\leq s,\\
 &\quad \mathrm{supp}(x)\subseteq \{1,2,\cdots,N\}.
 \end{split}
 \]
By introducing the  bound  of the support set  $J_1, J_2$,
a quadratic optimization problem was established as follows
\[
 \begin{split}
 \min\limits_{\bm x\in \mathbb R^n}&\quad \sum\limits_{i=1}^{N} (\bm x^T D_i \bm x-b_i^2)^2,\\
 \mbox{s.t.}, &\quad \|\bm x\|_0\leq s,\\
 &\quad J_1\subseteq\mathrm{supp}(x)\subseteq J_2,
 \end{split}
\]
 with $D_i=\Re(\mathcal F_i)^T\Re(\mathcal F_i)+\Im(\mathcal F_i)^T\Im(\mathcal F_i)$.
A restarted version of 2-opt method was designed to solve it, which consists of two steps, \emph{i.e.,} local search method to update the current support set, and Damped Gauss-Newton (DGN) method to solve a nonconvex optimization problem with the updated support set. In the first step to update the support set, a 2-opt method \cite{papadimitriou1998combinatorial} was used, and each time only two elements including one in the support and another in the off-support  were changed. Normally the 2-opt method can get stuck at local minimums. A weighted version of the objective functional with random generated weights in the above minimization problem solving by DGN was used to increase the robustness.

\subsection{The proposed models}
Both above methods require to know the sparsity, which may limit their applications for real signal recovery. Therefore, based on \eqref{pr_constrained}, we propose a novel variational model for PR without knowing the exact sparsity  (size of support set) in advance, which is stated as follows
\begin{equation}\label{l0lp}
\min\limits_{\bm x} ~~\mathcal G_\lambda(\bm x):=\lambda \|\bm x\|_0+\frac1p\big\| b-|\mathcal{A} \bm x|\big\|_p^p,
\end{equation}
where the parameter $\lambda$ controls the tradeoff between the sparsity and data fidelity, and $L^p$ norm measures the data fidelity ($ p \geq 1$).
%

We can directly establish the existence of the variation problems.
\begin{theorem}
Eqn. \eqref{l0lp} admits at least one minimizer if the general linear mapping $\mathcal A$ is continuous.
\end{theorem}
\begin{proof}
As the $L^0$ pseudo norm is  lower semi-continuous, and the objective functional $\mathcal G_\lambda(\bm x)$ in \eqref{l0lp} is lower bounded, one can readily prove it and we omit the detail here.
\end{proof}

The uniqueness is hardly to guaranteed since it involves a nonconvex optimization problem. However, by letting $\lambda=0$, the uniqueness of the proposed variational model reduces to the uniqueness of the PR problem. As \cite{chang2016}, in the absence of $L^0$ norm, one can prove the uniqueness of the PR problem. However, it is rather challengeable for the 1-D problem \cite{Shechtman2014} only from the Fourier measurements, and there may exist many nontrivial solutions. Thus, as the numerical experiments in section \ref{sec4} showed, the successful recovery rates are not guaranteed to be $100\%$  even when the sparsity is very strong and regularization is used, \emph{i.e.}, based on the minimization problem \eqref{l0lp}. If the uniqueness can not be guaranteed, we can only give the reliability of our proposed models for noisy free measurements in the following theorem.

First, we denote the solution set (or feasible set) for phase retrieval $\mathscr S(b):=\{\bm x\in\mathbb C^N:~|\mathcal A \bm x|=b\}$.
\begin{theorem}
Assume there exists  $\bm x^\star\in \mathscr S(b)\not=\emptyset$ with noisy free measurements $b$ which minimizes \eqref{l0lp}, it is the sparsest among the solution set  $\mathscr S(b)$, \emph{i.e.,} \[\bm x^\star=\arg\min\limits_{\bm x\in \mathscr S(b)} \|\bm x\|_0.\]
\end{theorem}
\begin{proof}
It can be readily proved since $\mathcal{G}_\lambda(\bm x)=\lambda\|\bm x\|_0$ for $\bm x\in \mathscr S(b)$.
\end{proof}

\begin{remark}
In the above theorem, we assume there exists  a minimizer to \eqref{l0lp} which belongs to the solution set $\mathscr S(b)$.
If such assumption does not hold,  for an arbitrary minimizer \[\tilde {\bm x}_\lambda=\arg\min \mathcal G_\lambda(\bm x),\] one has \[\tilde {\bm x}_\lambda\notin \mathscr S(b),\] i.e., \[\delta:=\frac1p\big\|b-|\mathcal A \tilde{\bm x}_\lambda|\big\|^p_p>0.\]
Then \[
\lambda\|\tilde{\bm x}\|_0+\delta  =\mathcal G_\lambda(\tilde{\bm x})<\mathcal G_{\lambda}(\bm x)=\lambda \|\bm x\|_0,
\quad\forall\bm x\in \mathscr S(b).\]
Therefore one has
\[
\|\tilde{\bm x}\|_0<\min\limits_{\bm x\in \mathscr S(b)}\|\bm x\|_0,
\]
and
\[
0<\delta<\lambda(\min\limits_{\bm x\in \mathscr S(b)}\|\bm x\|_0-\|\tilde{\bm x}\|_0)\leq \lambda n,
\]
 It implies that one can seek a  sparser approximation for the phase retrieval problem by setting the balance parameter $\lambda$ to be smaller enough, \emph{i.e.,} for arbitrary small $\varepsilon$, by selecting the parameter $\lambda\leq \frac{\varepsilon}{n}$, one can find a  sparser minimizer  $\bm x^\star$ to \eqref{l0lp}, which is very close to the solution set
with
\[
\frac{1}{p}\big\|b-|\mathcal A \bm x^\star|\big\|_p^p\leq \varepsilon.
\]
\end{remark}
\vskip .1in
\section{Numerical Algorithm for Proposed Model}\label{sec3}
We construct an alternating direction method  of multipliers (ADMM) with dynamical steps for solving \eqref{l0lp}. First, the minimization problem \eqref{l0lp} is rewritten as a constrained minimization problem by introducing a pair of new variables, which reads
\begin{equation}
\begin{split}
\min\limits_{\bm x,q,\bm z}& \quad\lambda \big\|\bm q\big\|_0+\frac1p\big\| b-|\bm z|\big\|^p_p\\
\mbox{ s.t.,}&\quad\bm x=\bm q,\quad \bm z=\mathcal A \bm x.
 \end{split}
 \label{lpcon}
\end{equation}
Then, the augmented Lagrangian is introduced in order to settle the following saddle-point problem
\begin{equation}\label{eqAl}
\begin{split}
\max\limits_{\bm \Lambda_1,\bm\Lambda_2}\min\limits_{\bm x,\bm q,\bm z}\mathscr L_{r_1,r_2}(\bm x,&\bm q,\bm z;\bm \Lambda_1,\bm\Lambda_2)=\lambda \big\|\bm q\big\|_0+\frac1p\big\| b-|\bm z|\big\|^p_p\\
&+\Re(\langle \bm x-\bm q,\bm\Lambda_1\rangle)+\frac{r_1}{2}\|\bm x-\bm q\|^2\\
&+\Re(\langle \bm z-\mathcal A\bm x ,\bm\Lambda_2\rangle)+\frac{r_2}{2}\|\bm z-\mathcal A\bm x\|^2,
\end{split}
\end{equation}
where $\bm\Lambda_1$, $\bm\Lambda_2$ are the Lagrange multipliers, $r_1$, $r_2$ are  two positive constants, and $\langle\cdot,\cdot\rangle $ represents the $L^2$ inner product.
%

We implement ADMM to solve \eqref{eqAl}, which is sketched as Algorithm \ref{alg1}. In the algorithm, we dynamically update $r_i$, $i=1,2$, which starts from a given value and is multiplied by $\rho$ ($\rho>1$) at each iteration.
\begin{algorithm}
\caption{\label{alg1} ADMM method for solving \eqref{eqAl}}
\textbf{Input:} $b$, parameters $\lambda$, $r_1^0$, $r_2^0$, $r_{max}$, $\rho$ and $n=0$; \\
\textbf{Initialization:} generalize random signals $\bm q^{0}$ and $\bm z^0$;\\
\textbf{Repeat}
\begin{itemize}
\item[(1)]
Solve \[
\bm x^{n+1}=\arg\min\limits_{\bm x}\mathscr L_{r_1^n,r_2^n}(\bm x,\bm q^n,\bm z^n;\bm \Lambda^n_1,\bm \Lambda^n_2);
\]
\item[(2)]
Solve \[
\bm q^{n+1}=\arg\min\limits_{\bm q}\mathscr L_{r_1^n,r_2^n}(\bm x^{n+1},\bm q,\bm z^n;\bm \Lambda^n_1,\bm \Lambda^n_2);
\]
\item[(3)]
Solve \[
\bm z^{n+1}=\arg\min\limits_{\bm z}\mathscr L_{r_1^n,r_2^n}(\bm x^{n+1},\bm q^{n+1},\bm z;\bm \Lambda^n_1,\bm \Lambda^n_2);
\]
\item[(4)]Update the Lagrangian multipliers
\[
\begin{split}
&\bm\Lambda_1^{n+1}=\bm\Lambda_1^n+r_1^n(\bm x^{n+1}-\bm q^{n+1});\\
&\bm\Lambda_2^{n+1}=\bm\Lambda_2^n+r_2^n(\bm z^{n+1}-\mathcal A\bm x^{n+1});\\
\end{split}
\]
\item[(5)] $r_1^{n+1}=\rho r_1^n$, $r_2^{n+1}=\rho r_2^n$, and $n:=n+1$;
\end{itemize}
\textbf{Until $r^{n+1}_1\geq r_{max}$}
\end{algorithm}
\begin{remark}
Note that the unknown $\bm x$ can only be found up to trivial degeneracies: time-shift, conjugate-flip and global phase-change since these operations on the signal do not affect  related autocorrelation. Therefore, we apply corresponding corrections to  trivial solutions in the numerical experiments.
\end{remark}
\subsection{Sub-minimization problems}
In the following, we discuss the solution to each sub-minimization problem.
\subsubsection{Sub-minimization problem with respect to $\bm x$}
The sub-minimization problem of $\bm x$ can be written as
\begin{equation}
\begin{split}
\min_{\bm x} ~ \Re(\langle\bm x, \bm\Lambda_1\rangle) &+\frac{r_1}2\|\bm x-\bm q\|^2\\
&-\Re(\langle\mathcal{A}\bm x, \bm\Lambda_2\rangle)+\frac{r_2}2\|\bm z-\mathcal{A} \bm x\|^2.
\end{split}
\label{sub_p}
\end{equation}
We shall compute the first order optimal condition for the above problem w.r.t  the complex variable $\bm x$. By separating the real and complex parts of the complex-valued variable as \cite{chang2016}, we are readily to obtain the following Euler-Lagrange equation of \eqref{sub_p}
\begin{equation}
(r_1\mathcal I+r_2\mathcal A^*\mathcal A) \bm x=r_1\bm q +r_2\mathcal{A}^*\bm z +\mathcal{A}^*\bm\Lambda_2-\bm\Lambda_1,
\label{el-x}
\end{equation}
where $\mathcal I$ is the identical operator if
\begin{equation}
\Im(\mathcal A^*\mathcal A)=0.
\label{adj}
\end{equation}
Obviously, there is the closed-form solution to \eqref{sub_p} for the following two kinds of linear operators satisfying \eqref{adj}.
One is that the operator $\mathcal{A}$ is the normalized discrete Fourier transformation, \emph{i.e.,} $\mathcal A^*\mathcal A=\mathcal I$.
In this case, we have
 \begin{equation}
\bm x=(r_1\bm q +r_2\mathcal{F}^*\bm z +\mathcal{F}^*\Lambda_2-\Lambda_1)/(r_1+r_2).
\label{el-x}
\end{equation}
The other one is when the operator $\mathcal A$ is defined as
\begin{equation}\label{cdp}
\mathcal A u=(\mathcal F (M_1\circ u),\mathcal F (M_2\circ u),\cdots,\mathcal F (M_K\circ u))
\end{equation}
 for coded diffraction pattern (CDP) \cite{candes2015phaseCDP}, where $\circ$ denotes the Hadamard product (a componentwise multiplication). In this case, we have
\[
\bm x=(r_1\bm q +r_2\mathcal{A}^*\bm z +\mathcal{A}^*\Lambda_2-\Lambda_1)/(r_1\mathcal I+r_2\sum\limits_{j=1}^K M^*_j\circ M_j),
\]
where $\cdot/\cdot$ denotes the componentwise division.
\begin{remark}
If the operator $\mathcal A$ does not satisfy \eqref{adj}, one needs to solve  more complicated equations, where the real and complex parts of the complex-valued variable is coupled, and see details in \cite{chang2016}.
\end{remark}

\subsubsection{Sub-minimization problem with respect to $\bm q$}
The sub-minimization problem of $\bm q$ is given as follows
\begin{equation}
\min_{\bm q} ~ \lambda \|\bm q\|_0 -\langle \bm q, \Lambda_1\rangle+\frac {r_1}2 \|\bm x-\bm q\|^2.
\label{sub_q}
\end{equation}

This subproblem is easy to solve because the energy functional \eqref{sub_q} can be spatially decomposed, where the minimization problem w.r.t. each element is performed individually and has a closed-form solution. Similar to \cite{duan2015regularized}, we can derive the optimal solution of $\bm q$ as follows
\begin{equation}
\bm q_i = \left \{
\begin{split}
 &0, \qquad\qquad \mbox{for~}i\in \left\{i:\Big(\bm x_i+\frac{(\bm\Lambda_1)_i}{r_1}\Big)^2\leq\frac{2\lambda}{r_1}\right\},
 \\
 &\bm x_i+\frac{(\bm\Lambda_1)_i}{r_1}, \qquad~ \mbox{otherwise}.
\end{split}
\right . \label{sol_q}
\end{equation}

\subsubsection{Sub-minimization problem with respect to $\bm z$}
The sub-minimization problem of $\bm z$ is given as follows
\begin{equation}
\min_{\bm z}\frac1p\big\|b-|\bm z|\big\|^p_p + \langle \bm z, \Lambda_2\rangle + \frac{r_2}2\|\bm z-\mathcal{A}\bm x\|^2.
\label{sub_z}
\end{equation}

We discuss the solution of $\bm z$ in two-fold based on the choices of $p$, \emph{i.e.,} $p=1$ and $p=2$, respectively.

\textbf{Case I:} The data fidelity is measured by $L^2$ norm

When $p=2$, the minimization problem \eqref{sub_z} becomes
\begin{equation}
\min_{\bm z}~\Big\{\frac12\big\|b-|\bm z|\big\|^2  + \frac{r_2}2\|\bm z- \bm W\|^2\Big\},
\label{sub_z_l2}
\end{equation}
where $\bm W = \mathcal{A}\bm x-\frac{{\Lambda }_{2}}{r_2}$. Similarly to \cite{wu2011augmented}, we use a geometric interpretation of the minimization problem \eqref{sub_z_l2} in Fig. \ref{geometric_interpretation}. It is shown that the minimizer $\bm z^\star$ is on the line $OW$, and each entry of the optimal value is with unit direction \[\mathrm{sign}(\bm W_i):=\dfrac{\bm W_i}{|\bm W_i|}.\]
\begin{figure}[!htb]
\begin{center}
\setcounter{subfigure}{0}
{\includegraphics[scale=0.5]{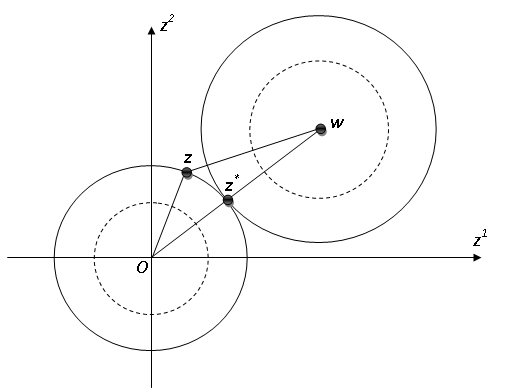}}
\caption{A geometric interpretation of the minimization problem \eqref{sub_z_l2}.} \label{geometric_interpretation}
\end{center}
\end{figure}
Therefore, the minimization procedure can be expressed as follows w.r.t. the length of each entry $k_i:=|\bm z_i|$
\begin{equation}
\min_{k_i\geq 0} ~\Big\{\frac12 \big(b_i-k_i|\bm W_i|\big)^2 + \frac{r_2}2 (1-k_i)^2|\bm W_i|^2\Big\}.
\label{min_k}
\end{equation}

The optimal solution of $k_i$ depends on the relationship of $b_i$ and $|\bm W_i|$, which is given as follows
\begin{equation}
k_i^\star=
\max\Big\{0, ~\frac{b_i+ r_2|\bm W_i|}{(1+r_2)\bm |\bm W_i|}\Big\},
\label{sol_z}
\end{equation}
which is degenerated when $|\bm W_i| = 0$. In this case, we can compute the optimal $k_i$ from
\begin{equation}
\min_{k_i\geq 0}~ \Big\{\frac12 \big(b_i - k_i\big)^2 + \frac{r_2}2 |k_i|^2\Big\},
\label{z_norm_problem}
\end{equation}
which gives us \[k_i^\star = \frac{b_i}{1+r_2}.\]

\textbf{Case II:} The data fidelity is measured by $L^1$ norm

When $p=1$, the minimization problem w.r.t. $\bm z$ becomes
\begin{equation}
\bm z^\star=\arg\min_{\bm z}~\Big\{\big\|b-|\bm z|\big\|  + \frac{r_2}2\|\bm z- \bm W\|^2\Big\}.
\label{sub_z_l1}
\end{equation}
Referred to Fig. \ref{geometric_interpretation}, the minimizer $\bm z^\star$ is still on the line $OW$. We pursue $\bm z^\star$ through solving the following minimization problem w.r.t. the length  $k_i$
\begin{equation}
\min_{k_i\geq 0}~\Big\{\big|b_i-k_i|\bm W_i|\big| + \frac{r_2}2 (1-k_i)^2|\bm W_i|^2\Big\}.
\label{min_k_l1}
\end{equation}
By letting $t= k\circ|\bm W|-b$, we have the minimization problem of $t$ as follows
\begin{equation}
\min~~\epsilon_i(t)= \Big\{|t_i| + \frac{r_2}2 (t_i - R_i)^2\Big\},~~\mbox{s.t.,}~ t_i\geq -b_i,
\label{min_t}
\end{equation}
where $R_i = |\bm W_i|-b_i$.

In fact, there exists the closed-form solution to the above minimization problem according to the following lemma.
\begin{lemma}
 \emph{Assume that}
 \begin{equation}
 y^\star = \arg \min_{y\geq y_1} \epsilon(y):=|y| + \frac{r}2 (y-y_0)^2~\mbox{with~}y_0\in\mathbb R.
 \label{min_l1_con}
 \end{equation}
 \emph{We have}
 \begin{eqnarray}
y^\star  =
\left \{
\begin{array}{ll}
 \hat y, & \text{if}~~y_1 \leq \hat y_0,
 \\
 y_1, & \text{if}~~ y_1 > \hat y_0~~\text{and}~~ y_0 \geq 0,\\
 0, & \text{if}~~ y_1 > \hat y_0~~\text{and}~~ y_0 < 0,
\end{array}
\right .
\label{min_l1_sol}
\end{eqnarray}
\emph{where $\hat y_0 = \mbox{sign}(y_0)\mbox{max} \{|y_0|-\frac1r,0\}.$}
\label{lemma1}
\end{lemma}
\begin{proof}
 We discuss the minimizer of \eqref{min_l1_con} in two cases as follows.
\begin{itemize}
\item[1)] If $y_1 \leq \hat y_0$, one readily has $y^\star = \hat y_0$, which is the soft thesholding of $y_0$ to minimize the problem without constraint $$\hat y_0=\arg\min\limits_y\epsilon(y).$$
\item[2)] If $y_1 > \hat y_0$, the minimizer relies on the sign of $y_0$. When $y_0 \geq 0$, after basic computation, the function $\epsilon (y)$ gains minimal values at $y_1$ if $y_1\geq y_0$ or $\hat y_0<y_1<y_0+\frac1r,$ \emph{i.e.}\[y^\star = y_1.\]
Otherwise when $y_0 < 0$, after basic computation similar to the above case, the function $\epsilon (y)$ gains minimal values at zero point, \emph{i.e.}\[y^\star = 0.\]
\end{itemize}
\end{proof}
According to   Lemma \ref{lemma1}, the above minimization problem \eqref{min_t} reaches its minimum $E^\star_i$ with
\begin{eqnarray}
t_i^\star  =
\left \{
\begin{array}{ll}
 \hat t_i, & \text{if}~-b_i \leq \hat t_i,
 \\
 -b_i, &\text{if}~ -b_i> \hat t_i~ \text {and} ~R_i \geq 0,\\
 0, &\text{if}~ -b_i> \hat t_i~\text {and}~R_i < 0,
\end{array}
\right.
\label{z_sol_l1}
\end{eqnarray}
where $$\hat t= \mathrm{sign}(R)\circ\mbox{max} \{|R|-\frac1{r_2},0\}.$$ Therefore, the optimal length is obtained as follows
\begin{equation}
k_i^\star=
\frac{t^\star_i +b_i}{|\bm W_i|}.
\label{sol_z_l1}
\end{equation}
\vskip .2in
To sum up the results of two cases with $p=1,2$, the optimal solution of $\bm z$ is given
\begin{equation}
\bm z^\star =  k^\star\circ \bm W,
\label{sol_zz}\
\end{equation}
where $k^\star$ is computed by \eqref{sol_z_l1} and \eqref{sol_z}  for $p=1,2$ respectively.


\subsection{Computational complexity}
The computational complexity for the proposed algorithm from Fourier measurements in each outer loop can be estimated as $$3N\log(N)+(29-2p)N,$$ where $N$ is the length of the signal and $p=1,2$ to represent different data fitting terms. It shows that the proposed algorithms is almost linear to the scale of the unknowns.
Further performances about the scalability will be reported in  the following section of experiments.

\subsection{Convergence analysis}
As discussed, one readily knows that each subproblem of the proposed algorithm are well defined. Therefore, we can directly  consider the convergence analysis. The problem is very challengeable, since the objective functional is nonconvex, and the regularized term is also non-differential. In order to provide the convergence analysis for nonconvex optimization problem as \cite{bolte2014proximal,wang2015global,hong2016convergence}, it requires the data fitting term has Lipschitz gradient. However, the data fitting term in our proposed model is not differential at all, which leads to the difficulty to obtain the convergence.

In order to give the convergence analysis, we need  an assumption that the iterative sequences of the multipliers are bounded.  Since the steps $r_j^n\rightarrow +\infty$ for $n\rightarrow +\infty$ for $j=1,2$, our assumption is  weaker than those in \cite{wen2012alternating,chang2016phase}, where the convergence for successive error of iterative multipliers is needed, \emph{i.e.,}
$$\lim\limits_{n\rightarrow \infty} (\Lambda_j^n-\Lambda_j^{n-1})= 0.$$
\begin{theorem}
Let $\bm \Psi^n:=(\bm x^n,\bm q^n,\bm z^n,\bm \Lambda_1^n,\bm\Lambda_2^n)$ be generated by Algorithm I. Assume that the two differences sequences of the multipliers $\{\bm\Lambda^n_j-\bm\Lambda^{n-1}_{j}\}$ for $j=1,2$ are bounded, and $\{\bm x^n\}$ is bounded.  Then there exists a accumulative point $\bm\Psi^\star=(\bm x^\star,\bm q^\star,\bm z^\star,\bm \Lambda_1^\star,\bm\Lambda_2^\star)$ of the sequence $\{\bm \Psi^n\}$, which satisfies the Karush-Kuhn-Tucker (KKT) conditions of \eqref{eqAl}, \emph{i.e.,}

\begin{equation}
\left\{
\begin{split}
&0=\Lambda_1^\star-\mathcal A^\star\bm\Lambda^\star_2,\\
&0\in \lambda\partial_{\bm q} \|\bm q^\star\|_0 -\bm\Lambda_1^\star,\\
&0\in \frac{1}{p}\partial_{\bm z} \|b-|\bm z^\star|\|_p^p+\bm\Lambda_2^\star,\\
&0=\bm q^\star-\bm x^\star,\\
&0=\bm z^\star-\mathcal A \bm x^\star.
\end{split}
\right.
\end{equation}

\end{theorem}

\begin{proof}
We just give a sketch proof, and see details in Refs. \cite{wen2012alternating,chang2016phase}.
First one needs to prove the boundedness of other iterative sequences.   By \eqref{sol_q}, one can derive the boundedness of $\bm q^n$. By \eqref{sol_z} for $p=1,2$, one can readily prove the $\bm z^n$ is bounded following the boundedness of $\bm x^n$ and $\bm \Lambda_2^n$. Therefore, $\bm\Psi^n$ is bounded, and as a result there exists a accumulative point $\bm \Psi^\star$. By the semi-continuities  of $L^0$ norm, we can finish the proof.
\end{proof}
\begin{remark}
In the work \cite{wen2012alternating}, a dynamic scheme was also adopted in the numerical tests. In this paper, we give a very simple updating schemes of the steps, which not only help to remove the assumption of the convergence of successive error sequences in \cite{wen2012alternating,chang2016phase}, but also improve the successful recovery rates dramatically compared with the fixed-step scheme by setting $\rho=1$. We will report the numerical performances in the following section.
\end{remark}
%
The above theorem only provides the local convergence of the proposed algorithm. However, as mentioned above, it is usually  difficult to establish the global convergence analysis for splitting type algorithms such as ADMM or proximal linearized algorithm since it can not  guarantee the sufficient decrease of the iterative errors.  As a future work, we aim to analyze the global convergence of ADMM for the nonconvex optimization problem with non-Lipschitz continuous gradient.

\section{Numerical Simulations}\label{sec4}
In this section, we use Algorithm 1 to solve the model \eqref{l0lp} with  $L^2$  and $L^1$ data fidelity, which are called as L0L2PR and L0L1PR, respectively. A series of numerical simulations are conducted to demonstrate the signal recovery accuracy, robustness to noise and computational efficiency.
All algorithms are implemented in Matlab R2013a environment on a Intel Core i5-4590 CPU@3.30GHz.

{\subsection{Implementation and Evaluation}
The choice of parameters in our model is quite easy, some of which can be fixed, \emph{i.e.,} $r_1^0$, $r_2^0$ and $r_{max}$. We give the default values of all parameters in TABLE \ref{parameters}. We tune the value of $\lambda$ relying on the sparsity and noise level of the observed magnitude data, which will be discussed in details later. Besides, for the signal with noise, we also decrease $\rho$ to $\rho=1.0001$ for better recovery accuracy. The notation $s$ is used to counter the number of nonzeros of sparse signals, and  bigger value of $s$ means the sparsity is weaker .
\begin{table}[ht]
\caption{\label{parameters} The default values of the parameters used in the proposed model.}
\begin{center}
\begin{spacing}{1.1}
\begin{tabular}{cccl}
\hline
\hline
Parameter & L0L2PR & L0L1PR&Interpretation\\
\hline
\hline
$\lambda$&$1.0\times 10^{-4}$&$1.0\times 10^{-3}$& control the sparsity\\
$r^0_1$&$1.0\times 10^{-3}$&$1.0\times 10^{-2}$&initial value of $r_1$\\
$r^0_2$&$1.0\times 10^{-3}$&$1.0\times 10^{-2}$&initial value of $r_2$\\
$\rho$&1.0005&1.0005&update speed of $\rho$\\
$r_{max}$&100&100&maximum of $r_1$, $r_2$\\
\hline
\hline
\end{tabular}
\end{spacing}
\end{center}
\end{table}

For a given random signal of length $N$ with $s$ nonzero elements, we calculate the magnitude of its DFT as $b$. We examine the recovery success rate called ``Recovery Probability'' as a function of the number of the successful trails, which is defined as \[
\mbox{Recovery Probability}=\dfrac{N_{suc}}{N_{trial}}\] with $N_{trial}$ trails.
In all tests, we set $N_{trail}=100$.
We compare our proposed algorithms with two state-of-art algorithms with relative low computational complexity, \emph{i.e.,} SPR algorithm \cite{mukherjee2014fienup} and GESPAR \cite{shechtman2014gespar}\footnote{\url{http://www.stanford.edu/~yoavsh/GESPAR_1D_Fourier_6_15.zip}}. For a fair comparison, we implement both SPR algorithm and GESPAR without support information. According to \cite{shechtman2014gespar}, the GESPAR is implemented with residual of iterative solution reaches the tolerance $TOL=10^{-4}$ and $ITER=6400$ for noiseless signal and $TOL=10^{-4}$ and $ITER=10000$ for noise signal, respectively.
}
\subsection{Performance for Noise Free Measurements}
In the first place, we examine the recovery success rate of the proposed algorithms.
We random generate a signal with complex values of length $N=128$ and sparsity $s=25$. The parameters in TABLE \ref{parameters} are adopted.
Both the measured Fourier magnitude $b$ and the recovered complex elements from L0L1PR are displayed in Fig. \ref{example}, which demonstrates the proposed algorithm work pretty well in this setting.

\begin{figure}[!htb]
\begin{center}
\subfigure[]{\includegraphics[scale=0.24]{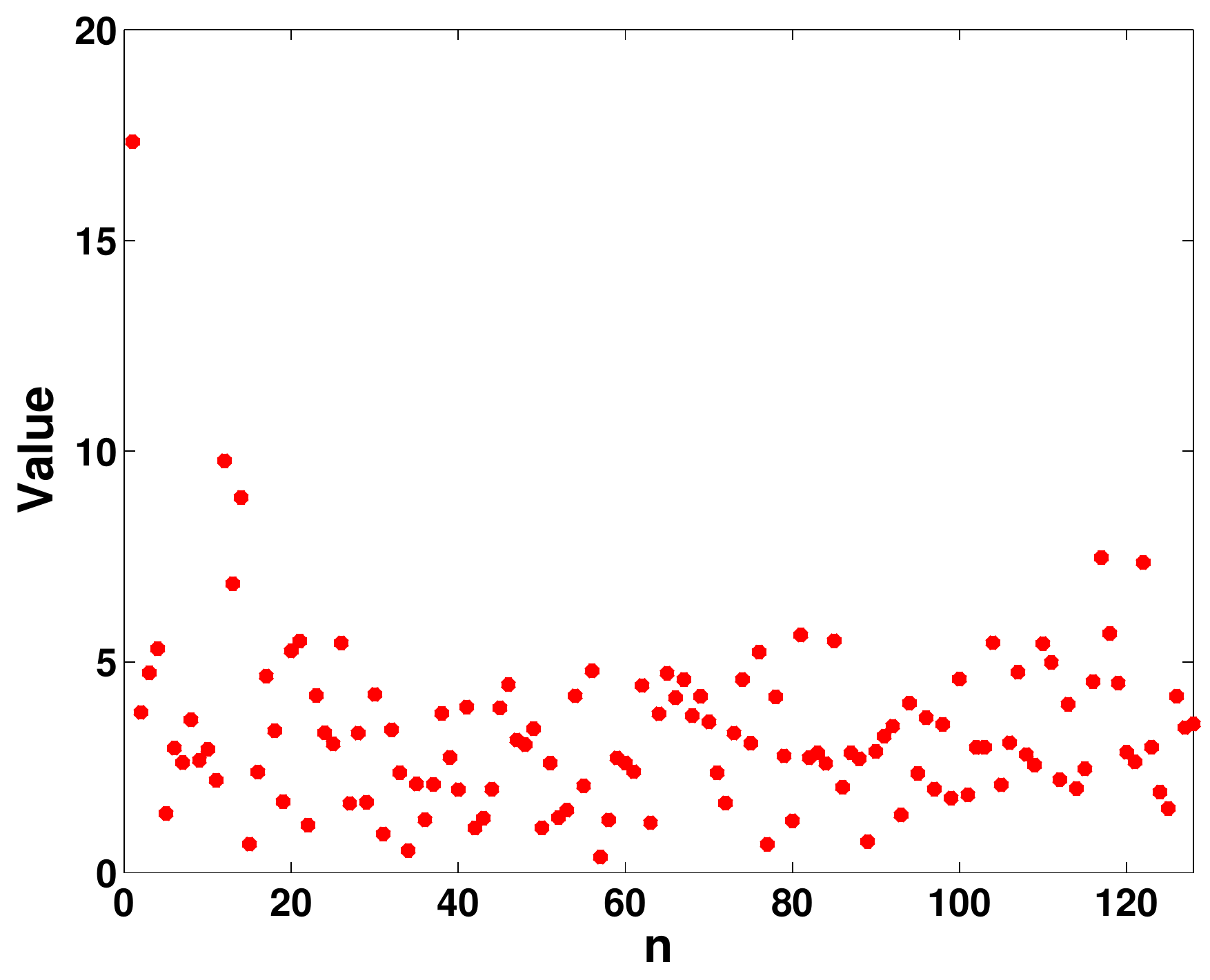}}
\subfigure[]{\includegraphics[scale=0.24]{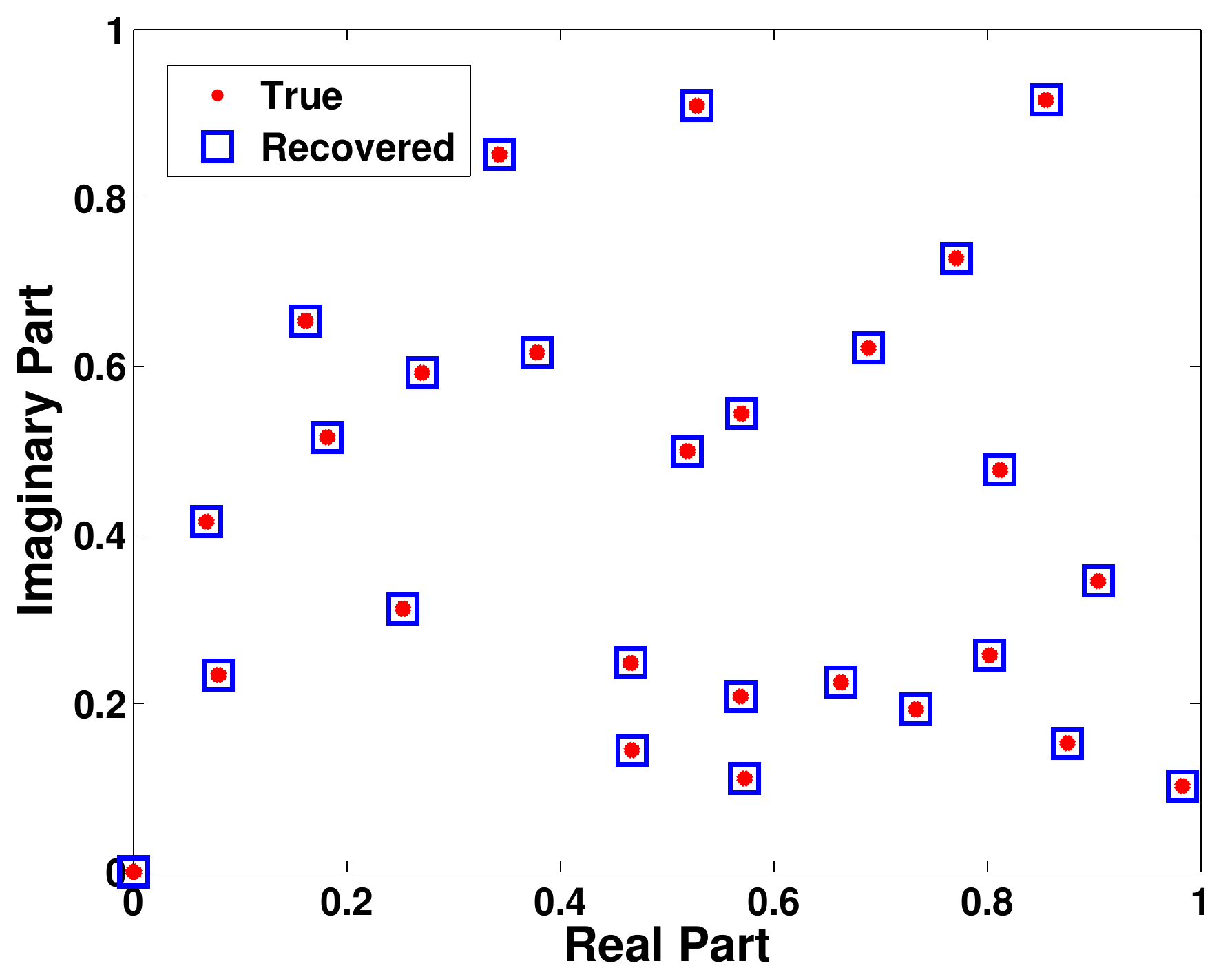}}
\caption{Fourier PR example. (a) Measured Fourier magnitude $b$. (b) True and recovered complex elements.} \label{example}
\end{center}
\end{figure}

We define the numerical energy for the proposed model as follows
\[
E(\bm q^n,\bm z^n) = \lambda\|\bm q^n\|_0 + \frac1p\|b-|\bm z^n|\|^p_p.
\]
Fig. \ref{example_energy} plots the numerical energies of the proposed algorithms terminated with $r_1>r_{max}$ for Fig. \ref{example}. Although the energies can not monotonically decreasing, and there are some oscillations, the energies of both two algorithms decay to steady states as the iterations increases.
\begin{figure}[!htb]
\begin{center}
\subfigure[]{\includegraphics[scale=0.24]{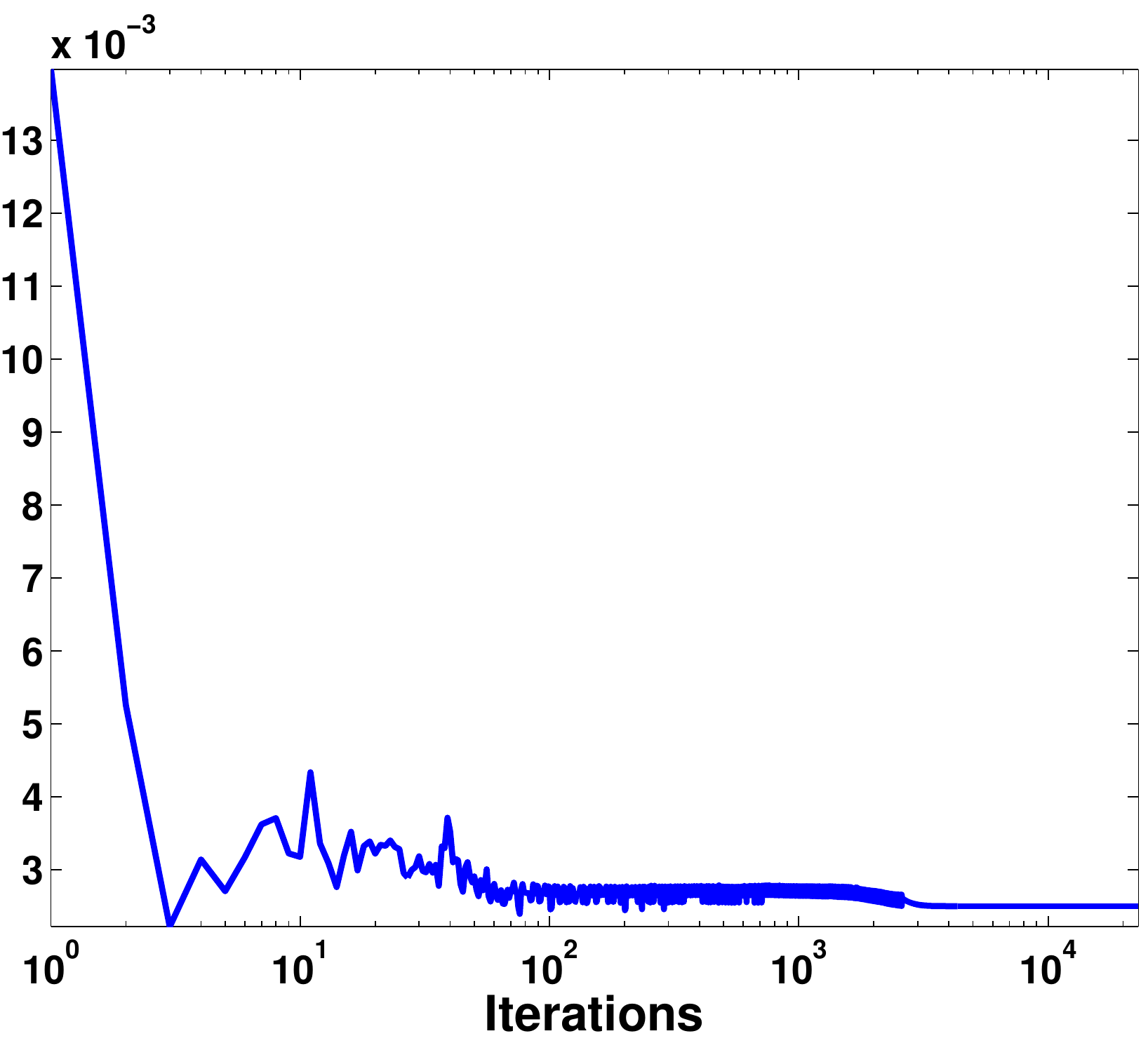}}
\subfigure[]{\includegraphics[scale=0.24]{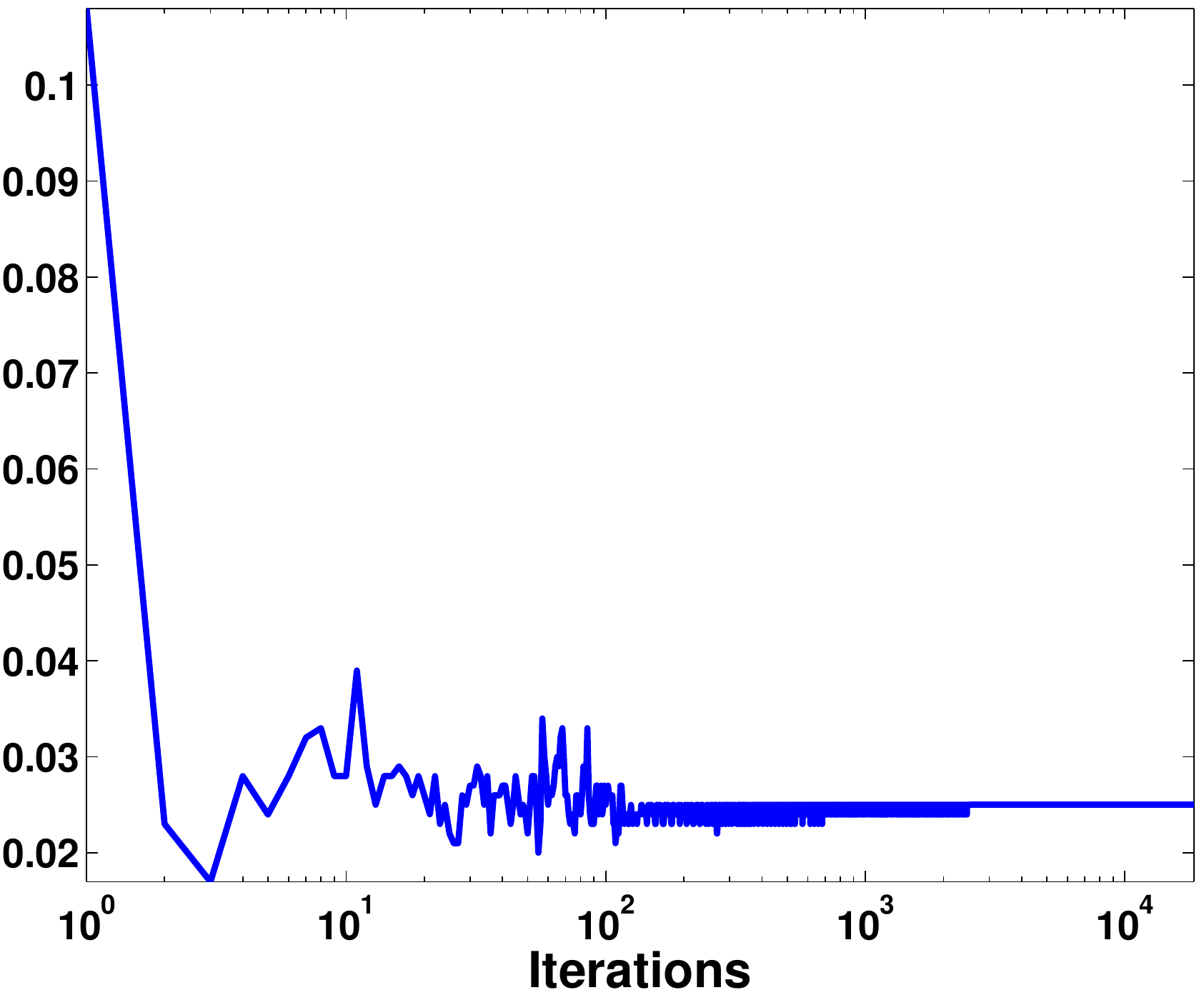}}
\caption{Decay of numerical energy of the 1D Fourier PR example. (a) Numerical energy of L0L2PR model. (b) Numerical energy of L0L1PR model.} \label{example_energy}
\end{center}
\end{figure}

Next, we compare the proposed algorithms with SPR algorithm and GESPAR. The recovery probability is shown in Fig. \ref{recovery_result_analysis}. We observe that all the algorithms work well when the sparsity is strong and the proposed algorithms outperforms the other two  when the sparsity becomes weak, \emph{i.e.,} $s\geq20$ for the signal of length $N=128$. It also demonstrates that recovery probability of L0L1PR is a little bit higher than that of L0L2PR for $s\in[25,30]$.

\begin{figure}[!htb]
\begin{center}
{\includegraphics[scale=0.3]{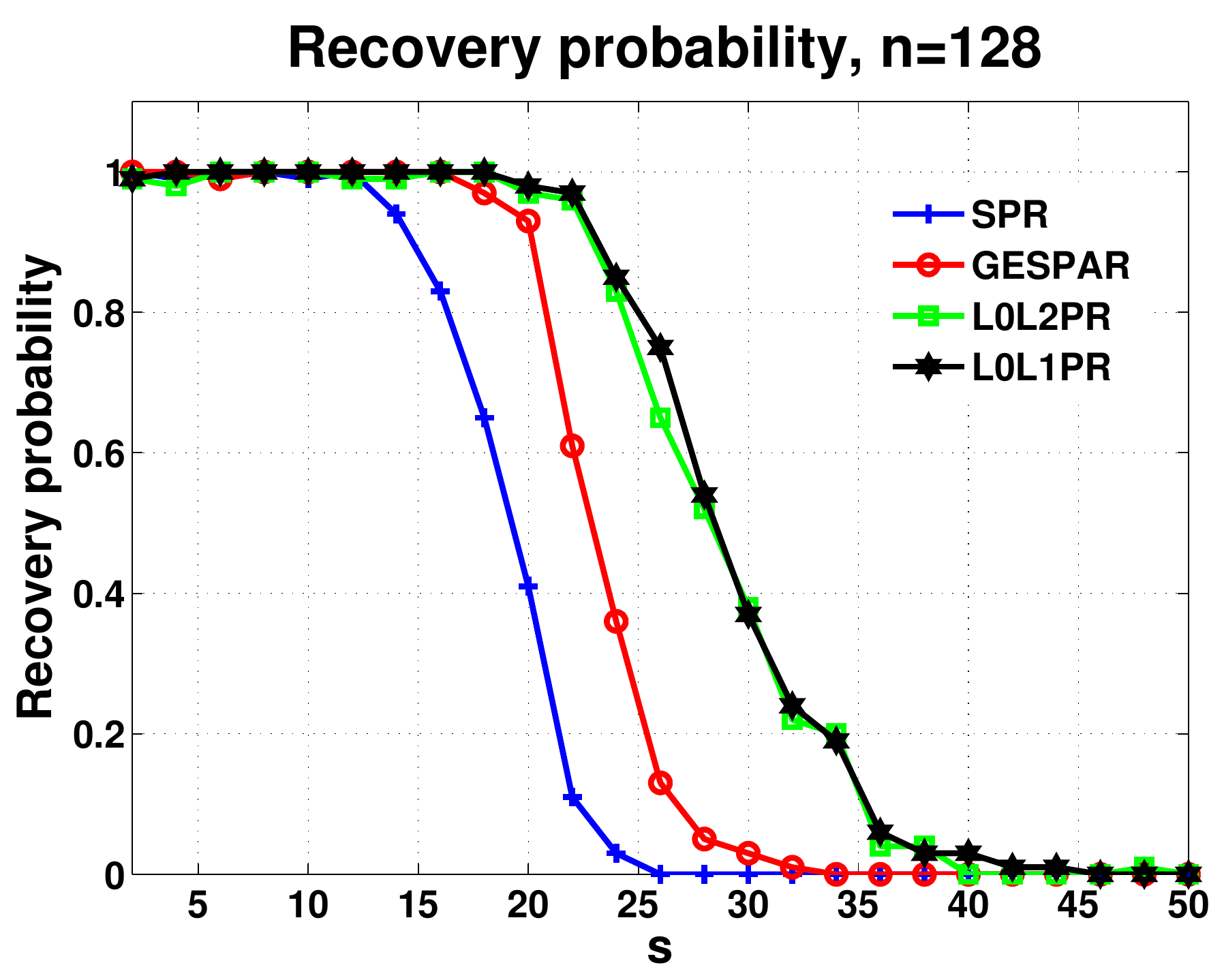}}
\caption{Recovery probability versus sparsity level.} \label{recovery_result_analysis}
\end{center}
\end{figure}

\subsection{Performance for Noisy Measurements}
In practice, the magnitude measurements may be corrupted by additive noise, \emph{i.e.,} the measurements are of the form
\begin{equation}
b = |\mathcal{F}\bm x|+ noise,
\label{noise_model}
\end{equation}
In our experiments, white Gaussian noise is added to the measurements at different SNR values which is defined below
\[\mbox{SNR}=-20\min\limits_{|c|=1}\log \frac{\|b-c \hat b\|}{\|\hat b\|},\]
where $b$ and $\hat b$ are clean and noisy measurements, respectively. In the implementation, we choose $\lambda$ for the proposed algorithms according to the noise level as shown in TABLE \ref{para_noise}, where larger $\lambda$ is used to recover signals with stronger noise, \emph{i.e.,} $\rho=1.0001$. Other parameters are set to the same as TABLE \ref{parameters}. The normalized mean squared error (NMSE) obtained from the proposed algorithms and GESPAR under different SNR values is plotted in Fig. \ref{noise_analysis}, which is explicitly defined as
\[
\mathrm{NMSE}=\min\limits_{|c|=1}\dfrac{\|\hat{\bm x}-c\bm x\|}{\|\bm x\|},
\]
with $\bm x,~\hat {\bm x}$ denoting clean and recovered signal, respectively. When the sparsity is very strong, e.g., $s=2$, GESPAR outperforms our proposed algorithms. When the sparsity decrease as $s$ goes to $s=20$, it is clearly shown that both $L^0$ based algorithms outperform GESPAR. The proposed L0L1PR and L0L2PR can  produce comparable results from the noisy measurements, and are quite robust w.r.t. different noise levels.
\begin{table}[ht]
\caption{\label{para_noise} $\lambda$ used for L0L2PR and L0L1PR under different SNR.}
\begin{center}
\begin{tabular}{cccc}
\hline
\hline
SNR&40&30&20\\
\hline
\hline
L0L2PR&$1.0\times 10^{-4}$&$5.0\times 10^{-4}$&$3.0\times 10^{-3}$ \\
L0L1PR&$2.0\times 10^{-2}$&$8.0\times 10^{-3}$&$1.5\times 10^{-3}$\\
\hline
\hline
\end{tabular}
\end{center}
\end{table}
\begin{figure}[!htb]
\begin{center}
{\includegraphics[scale=0.3]{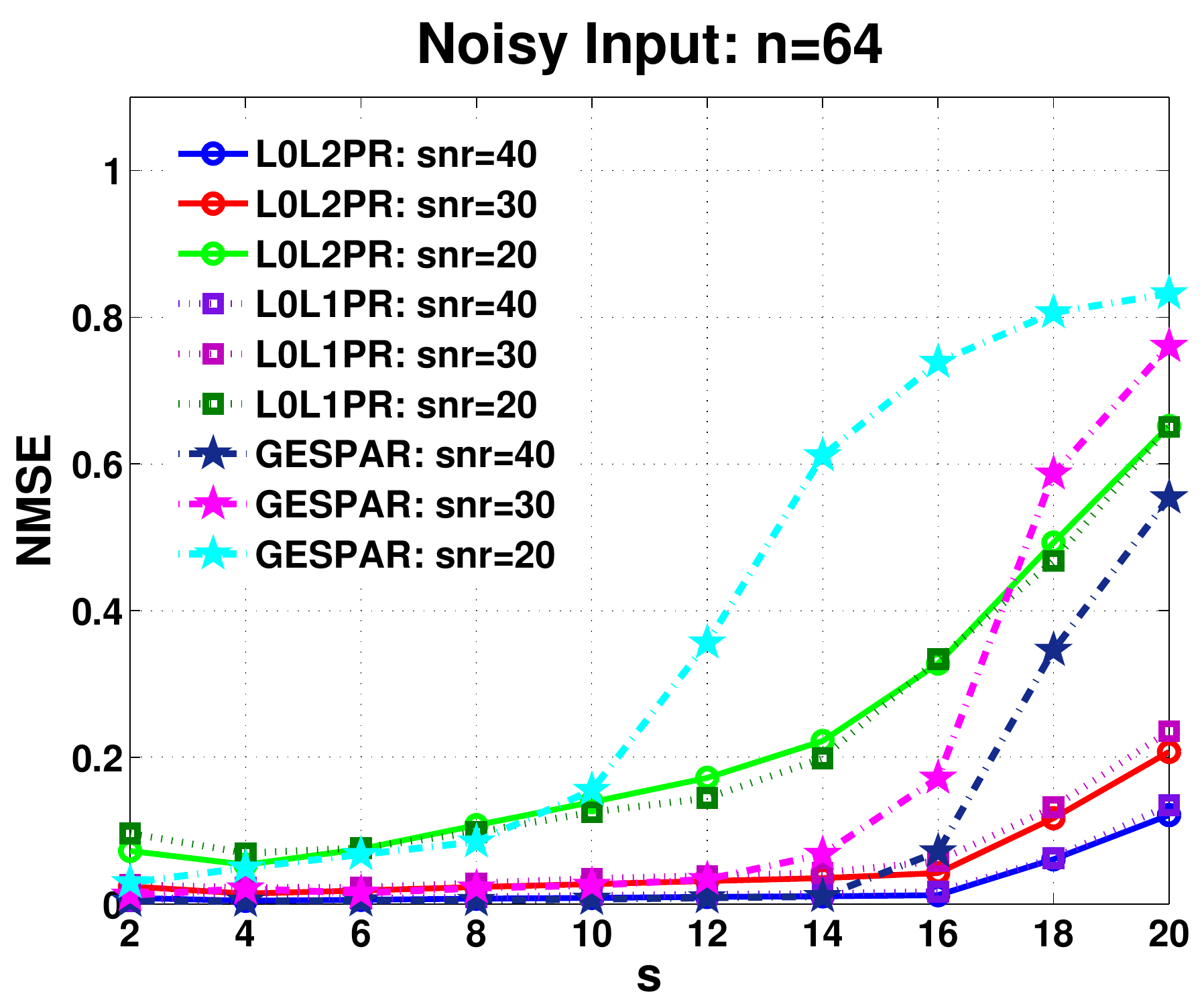}}
\caption{Normalized MSE versus sparsity level.} \label{noise_analysis}
\end{center}
\end{figure}

\subsection{Scalability}
The most significant advantage of the $L^0$ based algorithms over SPR methods and GESPAR is their ability to recover the signal with weak sparsity. We now examine the performances for signals with different sparsity levels. The recovery probabilities of the proposed algorithms are collected  by conducting evaluations on signals of length with $N=3200$, $N=6400$ and $N=12800$ and the sparsity $s$ among $s=30$ to $s=180$. For this experiment, we set the noise level to $SNR=1001$, which can be regarded as the noiseless case.

The recovery probability of L0L2PR versus different lengths of signal is plotted in Fig. \ref{scalabilityl2}. The subplot of Fig. \ref{scalabilityl2} presents the recovery probability of L0L2PR with the same regularization parameter $\lambda=1.0\times 10^{-5}$, in which the recovery probability of signal with length $N=12800$ is not good enough when the sparsity is strong. Indeed, we can improve the recovery probability by tuning optimal $\lambda$ w.r.t. different sparsity levels as the main plot shown in Fig. \ref{scalabilityl2}.
More importantly, we observe that the maximal $s$ that can be recovered successfully increases when the length of signal $N$ increases.
\begin{figure}[!htb]
\begin{center}
\includegraphics[width=.3\textwidth]{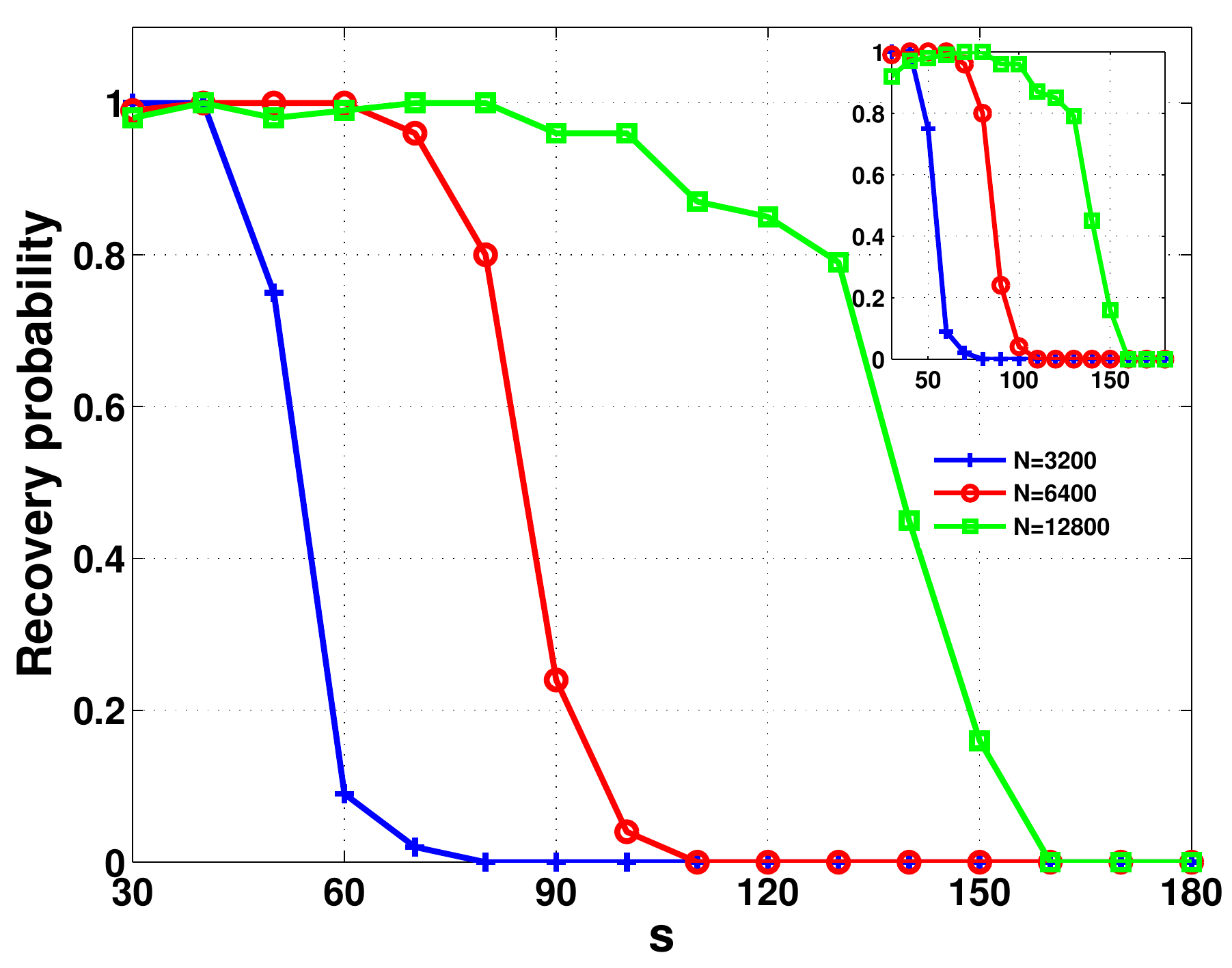}
\caption{Recovery probability of L0L2PR versus sparsity level.} \label{scalabilityl2}
\end{center}
\end{figure}

Similarly, we set $\lambda=5.0\times 10^{-5}$ for L0L1PR and plot the recovery probability for signals of different lengths in Fig. \ref{scalabilityl1}. By comparing with Fig. \ref{scalabilityl2}, we observe that the recovery probability of L0L1PR is slightly better than L0L2PR when $s\leq120$, and slightly worse than L0L2PR when $120<s<150$. It means that the $L^1$ data fidelity outperforms the $L^2$ data fitting when the sparsity is significant.
\begin{figure}[!htb]
\begin{center}
\includegraphics[width=.3\textwidth]{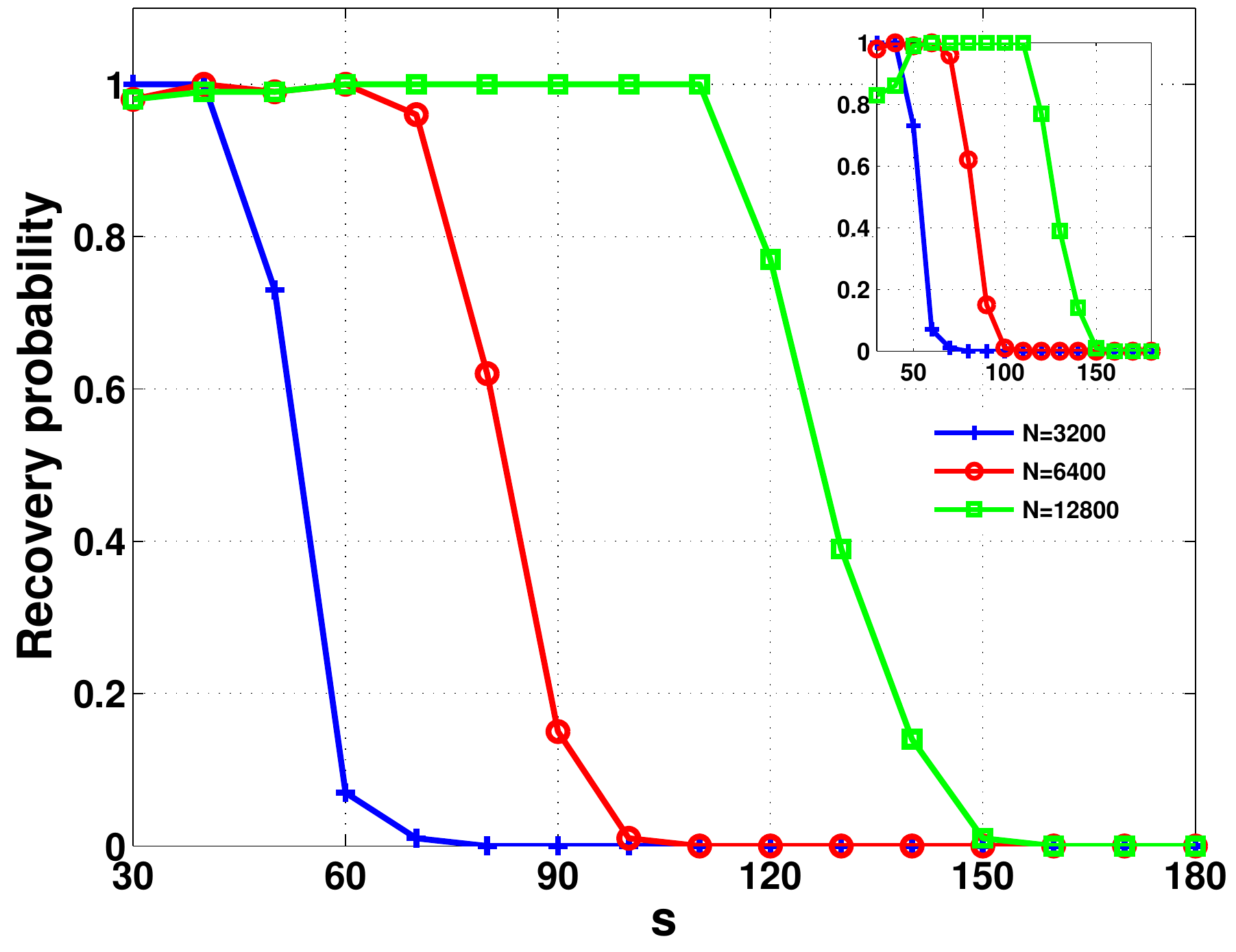}
\caption{Recovery probability of L0L1PR versus sparsity level.} \label{scalabilityl1}
\end{center}
\end{figure}

In addition, we also conduct a comparison experiment among SPR, GESPAR and L0L1PR by decreasing the sparsity of signals with length from $N=16$ to $N=1024$. We define the sparsity ratio (SR) as
\[\mathrm{SR}=\dfrac{s}{N}\times 100\%.\]
In the experiments, we set SR to $2\%$, $4\%$, $6\%$, $8\%$ and $10\%$ of the signal length, respectively. Selected recovery probability and average runtime comparison of the three algorithms is shown in TABLE \ref{comparison} for $N=1024$. The runtime is averaged over all successful recoveries. As shown in the table, SPR based algorithm is significantly faster than GESPAR and L0L1PR. The GESPAR is faster than the L0L1PR when the sparsity is strong, \emph{i.e.,} nonzero elements are limited, while the efficiency of GESPAR drops significantly as the number of nonzero elements increases. Besides, L0L1PR can recover the signal with a recovery probability $62\%$ when the recovery probability of both SPR and GESPAR already drops to zero as the number of nonzero elements increases to $s=82$. The recovery probability versus sparsity for different lengths of signal is shown in Fig. \ref{scalability}, which demonstrates that GESPAR outperforms the others for signals with very strong sparsity, and L0L1 outperforms the others for signals with relatively weak sparsity.

\begin{table*}[ht]
\caption{\label{comparison} Recovery probability (\%) and runtime (sec) comparison.}
\begin{center}
\begin{spacing}{1.1}
\begin{tabular}{|c|c|c|c|c|c|c|c|}
\hline
Sparsity Ratio& Nonzero Number &\multicolumn{2}{c|}{SPR}&\multicolumn{2}{c|}{GESPAR}&\multicolumn{2}{c|}{L0L1PR}\\
\hline
\hline
SP&s&recovery&time&recovery&time&recovery&time\\
\hline
\hline
2\%&21&100&0.001&100&2.46&100&6.85\\
4\%&41&100&0.001&100&2.46&100&6.85\\
6\%&62&75&0.002&21&346.37&93&6.98\\
8\%&82&0&--&0&--&62&6.98\\
10\%&103&0&--&0&--&27&7.01\\
\hline
\end{tabular}
\end{spacing}
\end{center}
\end{table*}

\begin{figure}[!htb]
\begin{center}
\subfigure[SPR]{\includegraphics[width=.3\textwidth]{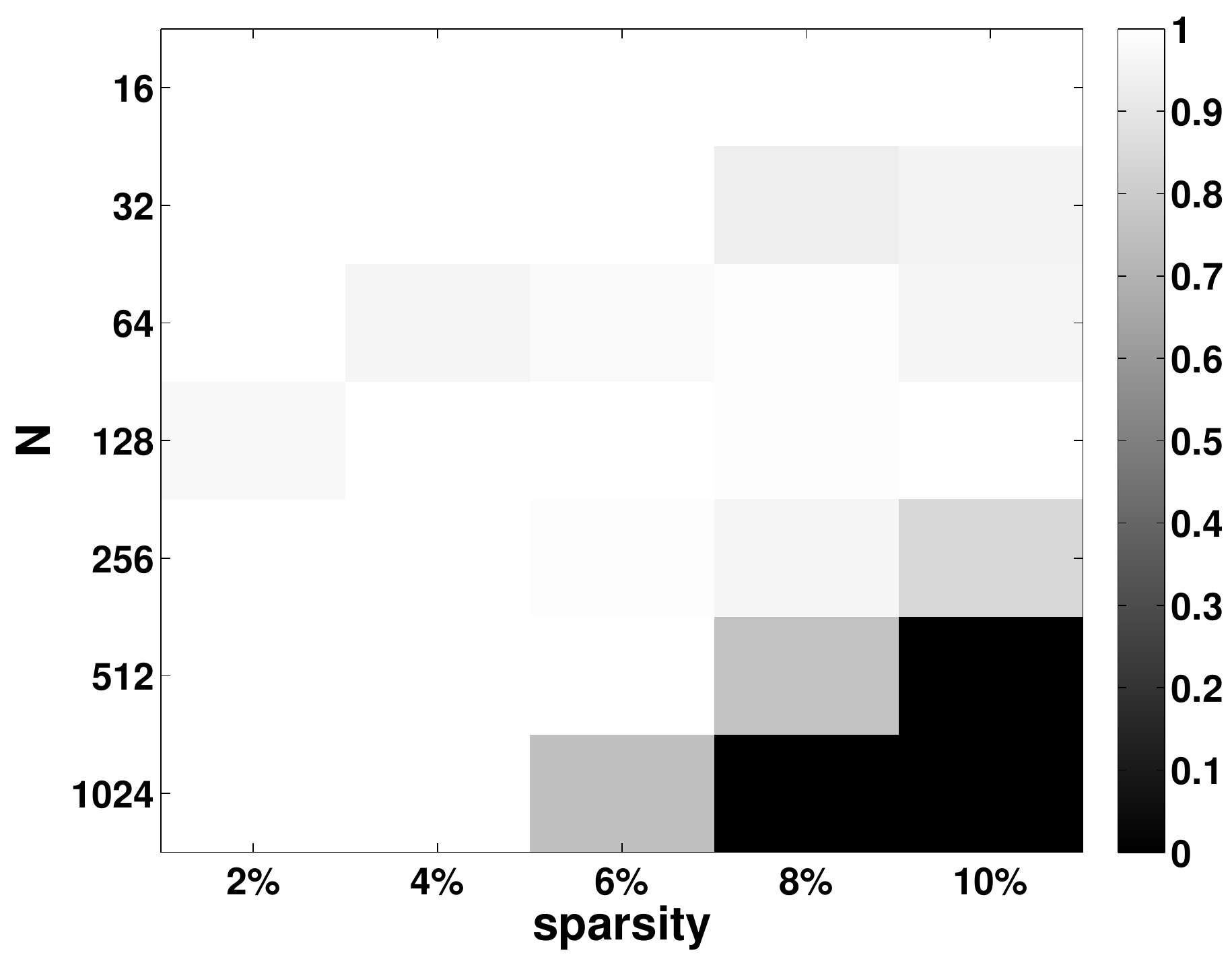}}
\subfigure[GESPAR]{\includegraphics[width=.3\textwidth]{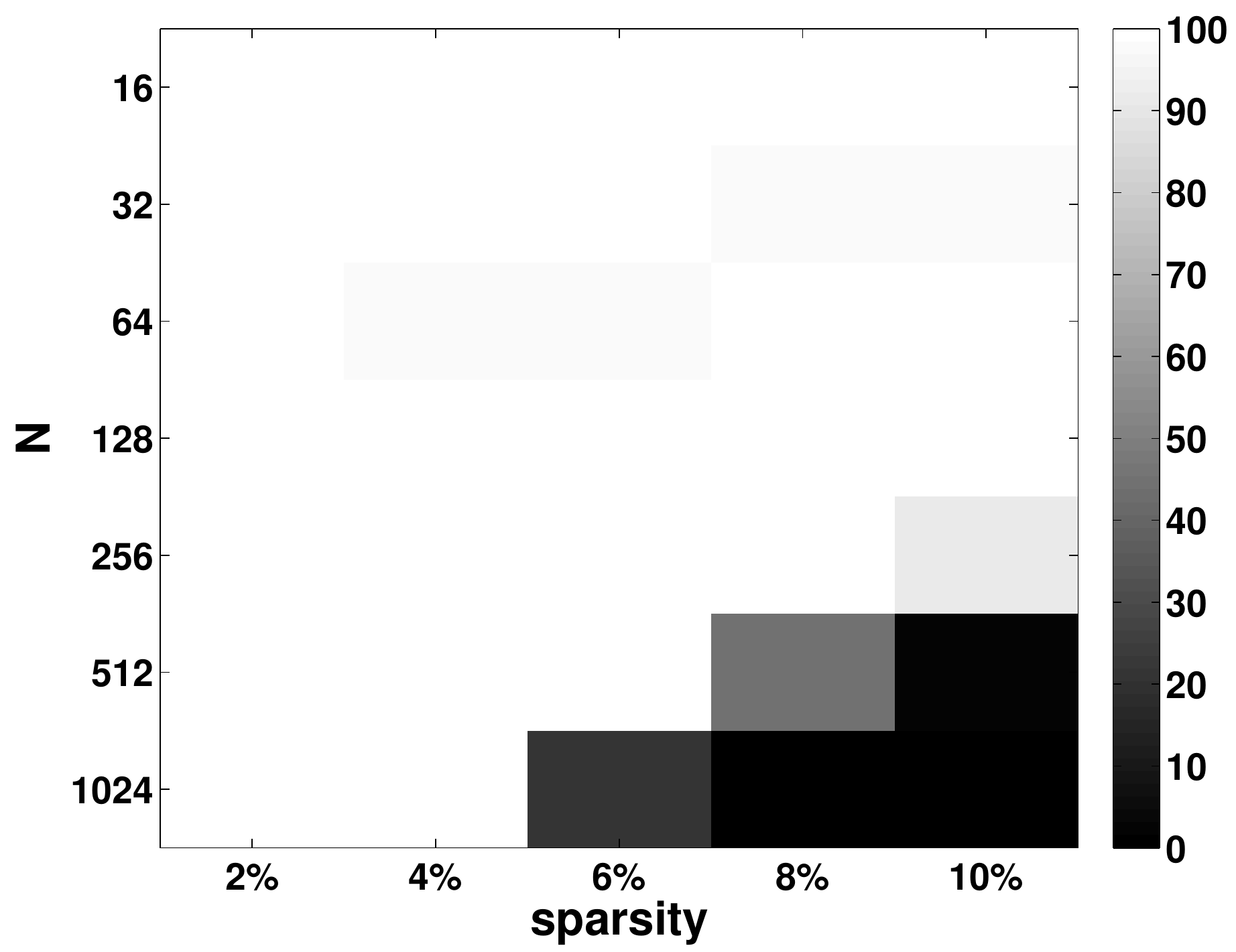}}
\subfigure[L0L1PR]{\includegraphics[width=.3\textwidth]{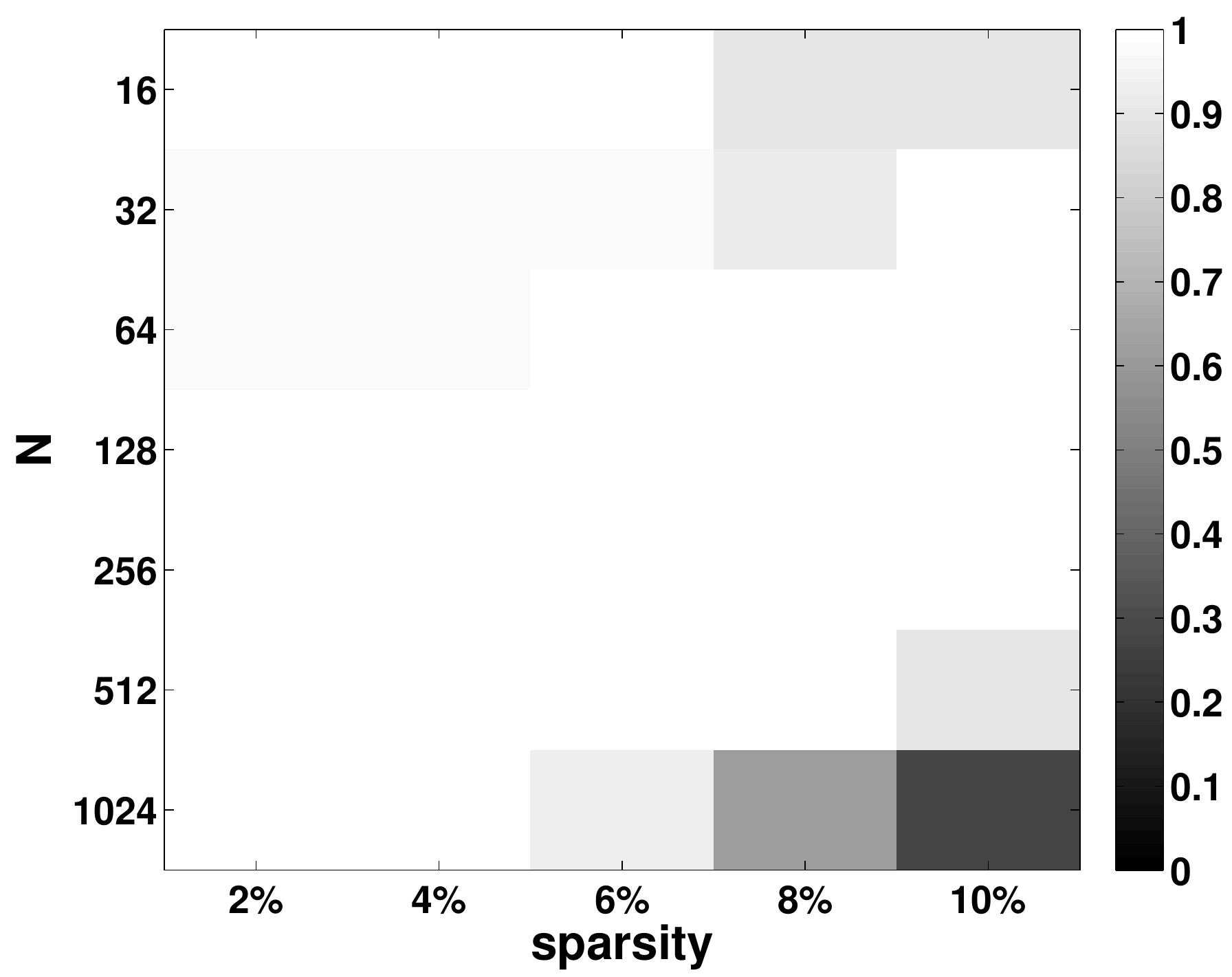}}
\caption{Recovery probability versus sparsity. (a) SPR; (b) GESPAR; (c) L0L1PR} \label{scalability}
\end{center}
\end{figure}

\begin{remark}
For a fair comparison, we fix $\lambda=1.0\times 10^{-3}$ for all testing signals with different lengths and sparsity. As shown by the previous experiment, the recovery probability of signals with $N=16$, $N=32$ and $s=8\%$, $s=10\%$ can be further improved by tuning the parameter $\lambda$.
\end{remark}

\subsection{Computational Time}
We compare the computational efficiency of the proposed $L^0$ based model and GESPAR in two-fold. On one hand, we work on signals with fixed length $N=512$ and variable sparsity, \emph{i.e.,} $\mathrm{SR}=2\%,~4\%,~ 6\%,~8\%,~10\%$. Both computational time and the corresponding NMSE are plotted in Fig. \ref{time1}. It shows that the computational time of $L^0$ based algorithms remain the same even when the sparsity of the signals decreases, while the computational time of GESPAR increases dramatically when the nonzero elements of the signals increases. Moreover, the NMSE also demonstrate that the proposed L0L1PR can produce higher accuracy  results. \begin{figure}[!htb]
\begin{center}
\subfigure[]{\includegraphics[width=.24\textwidth]{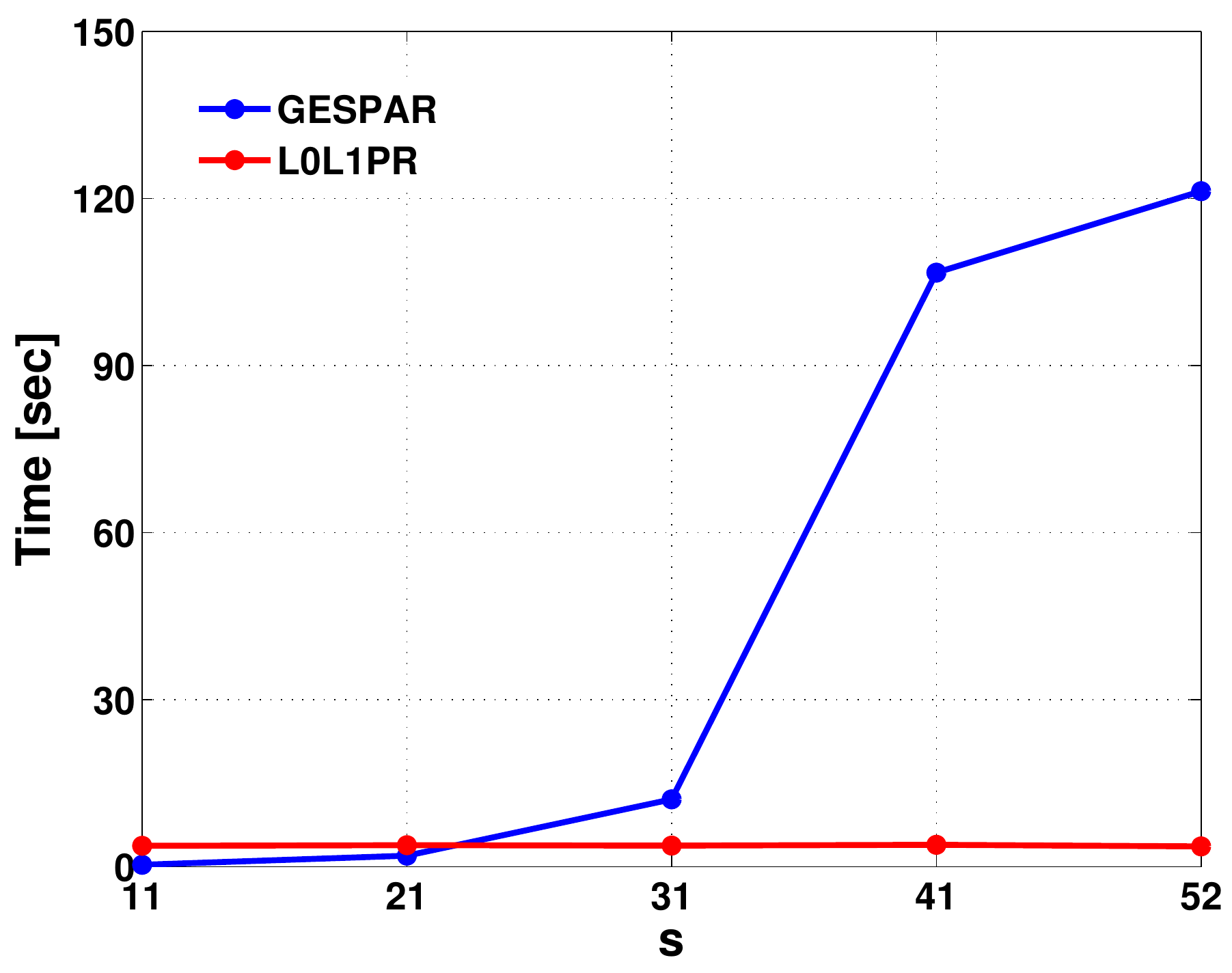}}
\subfigure[]{\includegraphics[width=.24\textwidth]{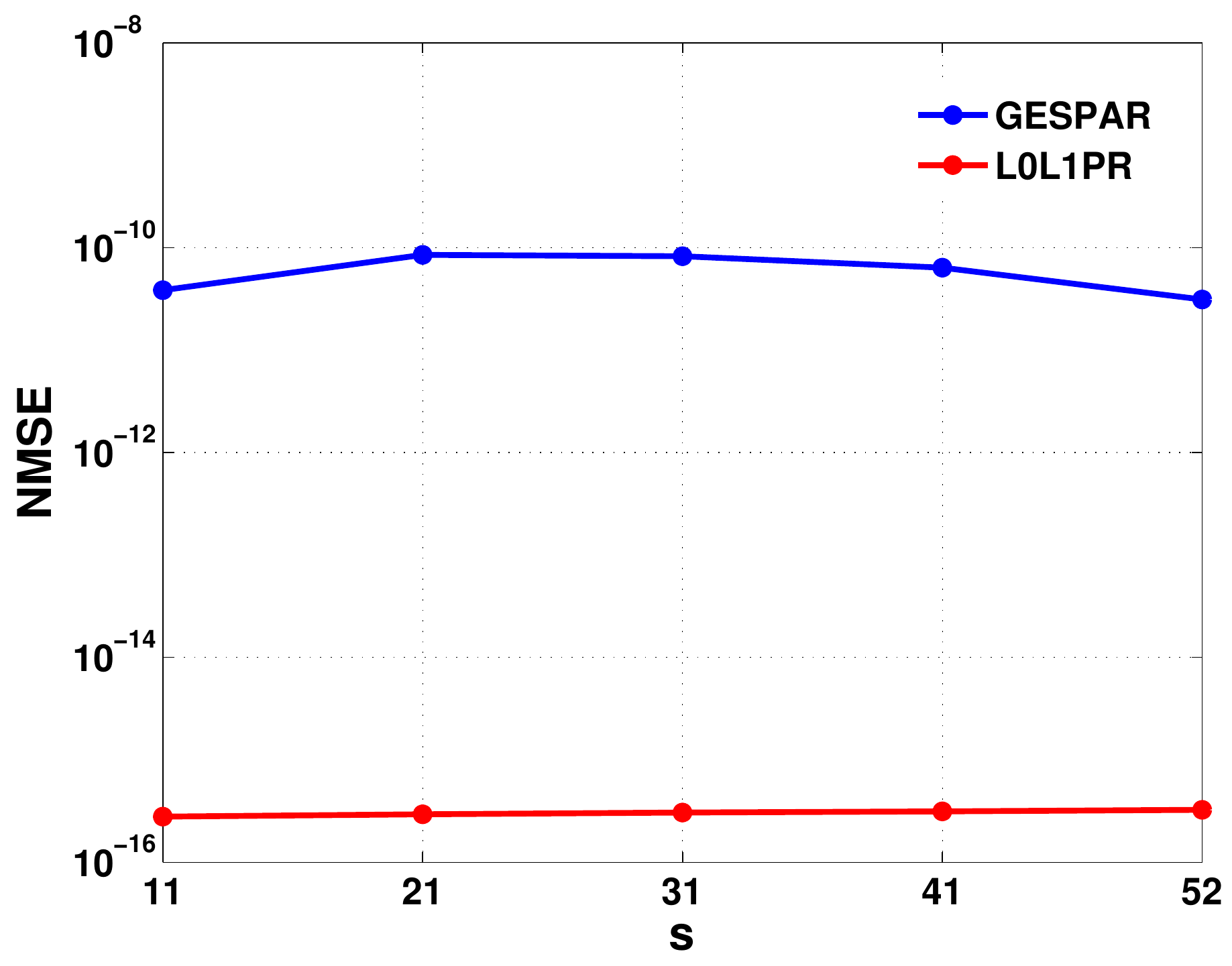}}
\caption{Computational time versus sparsity level. (a) Computational time in second. (b) Normalized mean squared error.} \label{time1}
\end{center}
\end{figure}

On the other hand, we fix the sparsity ratio $\mathrm{SR}=8\%$ w.r.t. signals of different lengths, \emph{i.e.,} $N=16,~32,~64,~\cdots,~1024$. Both computational time and NMSE are plotted in Fig. \ref{time2}, which also demonstrates that the proposed algorithms are quite suitable for large scale problems.

\begin{figure}[!htb]
\begin{center}
\subfigure[]{\includegraphics[width=.24\textwidth]{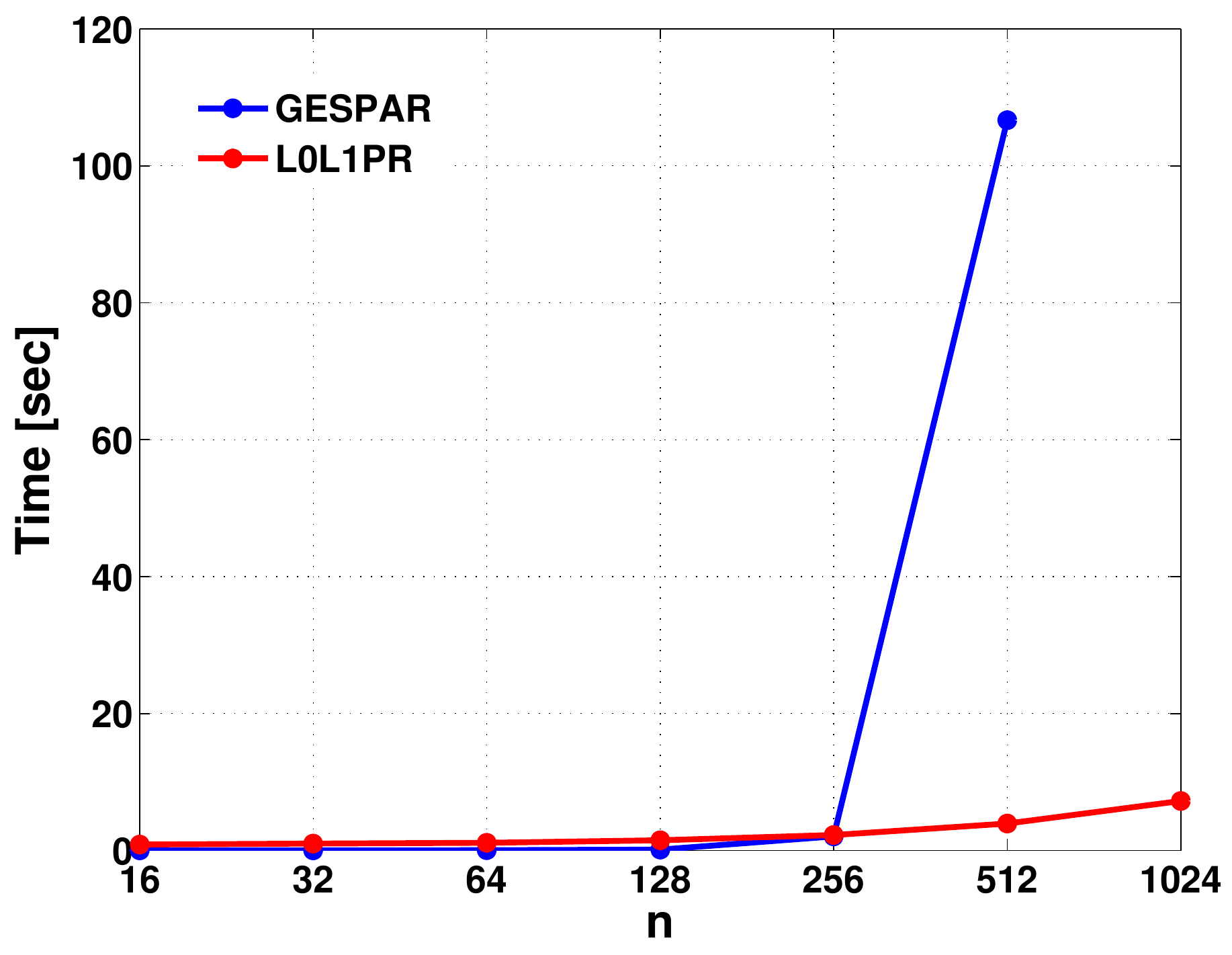}}
\subfigure[]{\includegraphics[width=.24\textwidth]{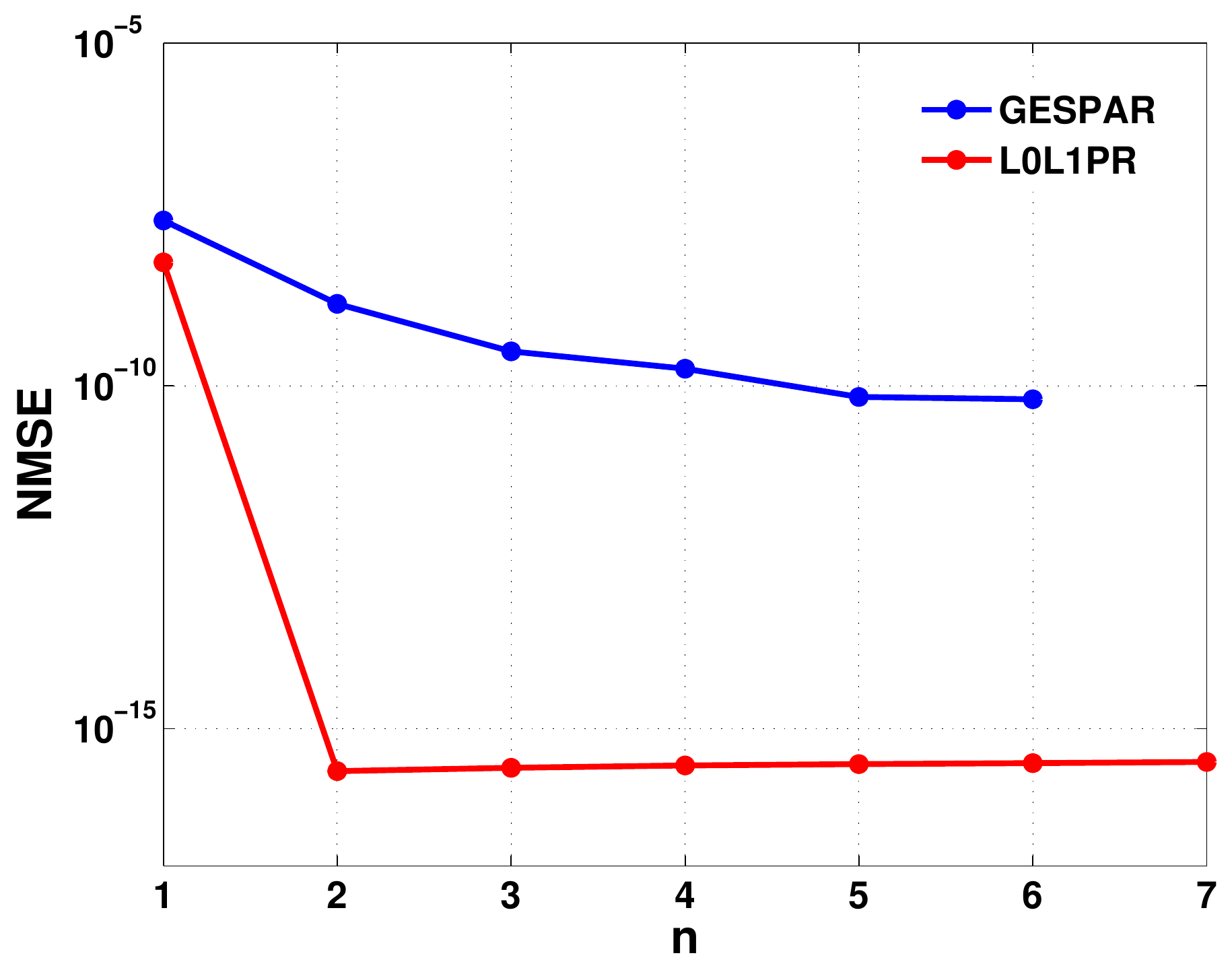}}
\caption{Computational time versus sparsity level. (a) Computational time in second. (b) Normalized mean squared error.} \label{time2}
\end{center}
\end{figure}



\subsection{Discussion on the Parameters}
The most important parameter in the proposed model is $\lambda$, which controls the sparsity of the recovery results. We analyze the impact of $\lambda$ for L0L2PR and L0L1PR by using different values of $\lambda$ on the same input signals and plot the recovery probability in Fig. \ref{lambda1}. It is shown that in order to gain higher recovery probability, a moderate  $\lambda$ is needed. Otherwise, the convergence speed and recovery probability both decrease.
\begin{figure}[!htb]
\begin{center}
\subfigure[]{\includegraphics[width=.24\textwidth]{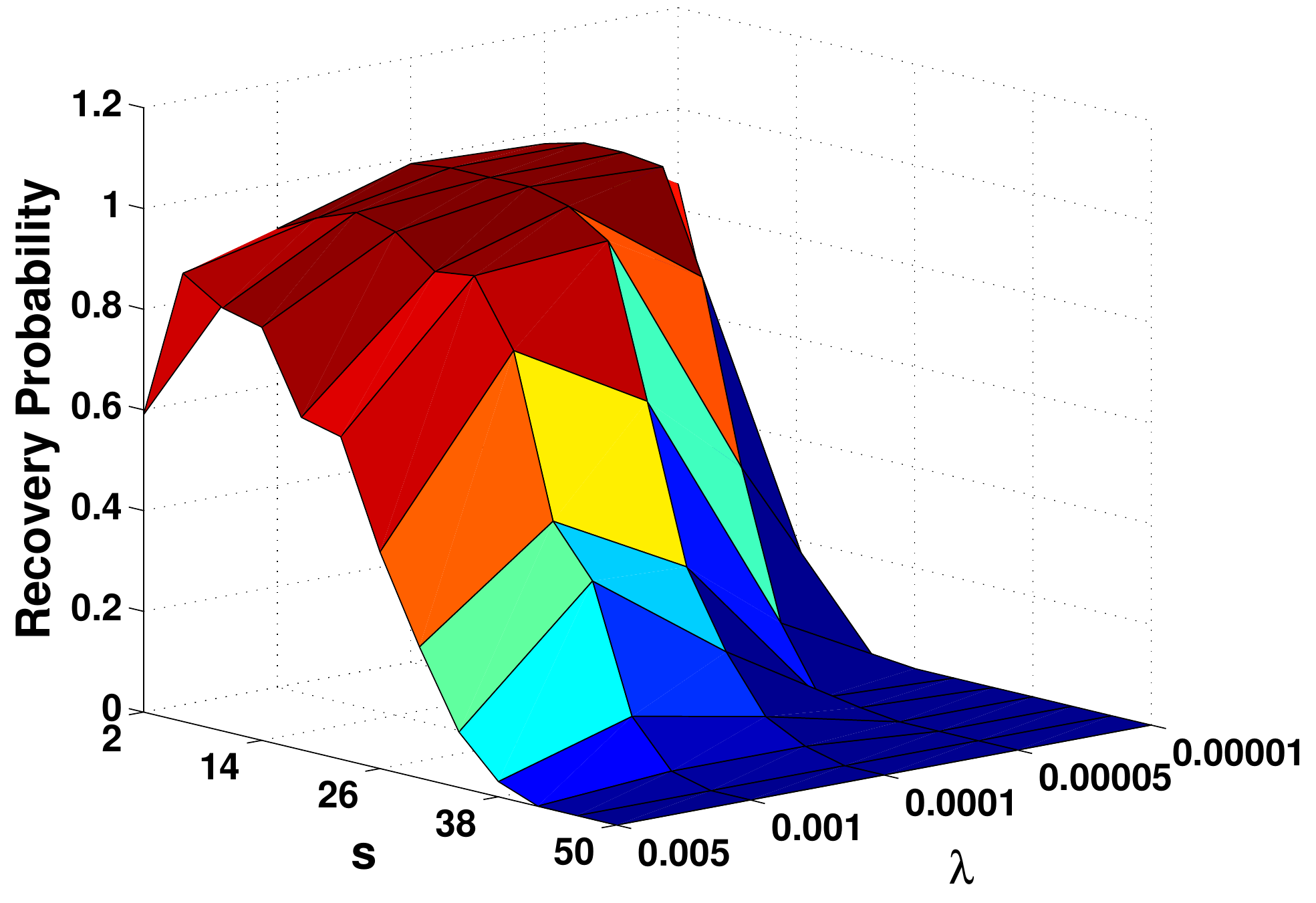}}
\subfigure[]{\includegraphics[width=.24\textwidth]{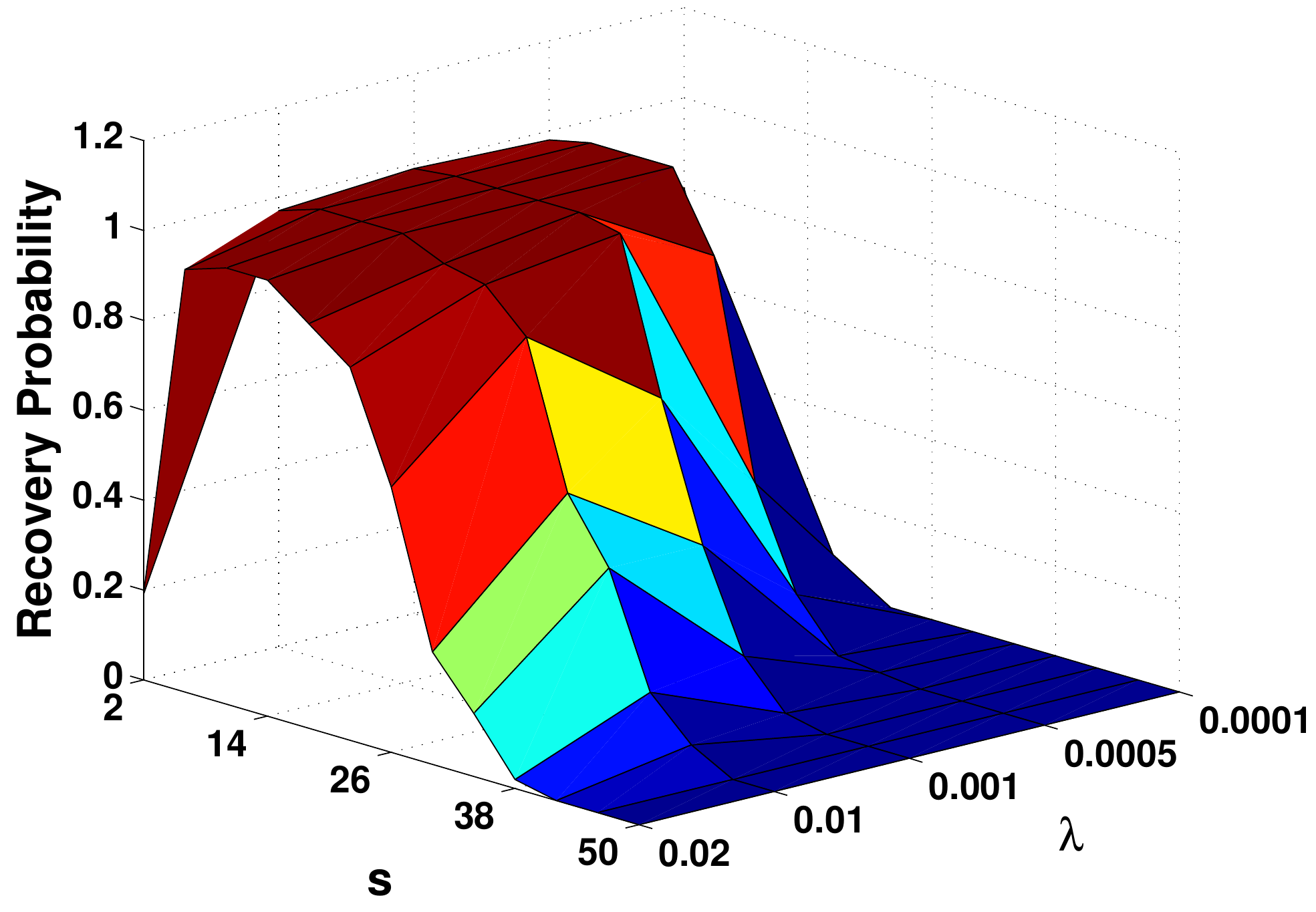}}
\caption{Regularization parameter $\lambda$ versus sparsity level. (a) L0L2PR; (b) L0L1PR} \label{lambda1}
\end{center}
\end{figure}


In Algorithm 1, we used the dynamic steps for better performance by setting $\rho>1$. In order to demonstrate the advantage of the dynamic schemes, we compare the recovery probability for different sparsity $s$ using L0L1PR with or without dynamic steps, which is plotted in Fig. \ref{figstep}. The parameters and stopping condition for L0L1PR with dynamic scheme are kept the same as the above tests, while we set $\lambda=1.0\times 10^{-2}, r_1^0=r_2^0=0.5$ for the fixed-step version algorithm with $\rho=1$.  Same number of iterations  are adopted for these two compared algorithms.  With fixed steps, for sparsity $s<20$, the recovery probability is only up to 80\%, while with dynamic step, it is about 100\%. When the sparsity decreases, fixed-step algorithm  hardly  recover the signals with $s\geq 26$ while dynamic-step algorithm can still recover the signals with $s\geq 40$ with low successful rate.  It is obvious that our proposed algorithm with dynamic steps  can efficiently improve the recovery probability compared that with fixed-step.
\begin{figure}[!htb]
\begin{center}
\subfigure[]{\includegraphics[width=.24\textwidth]{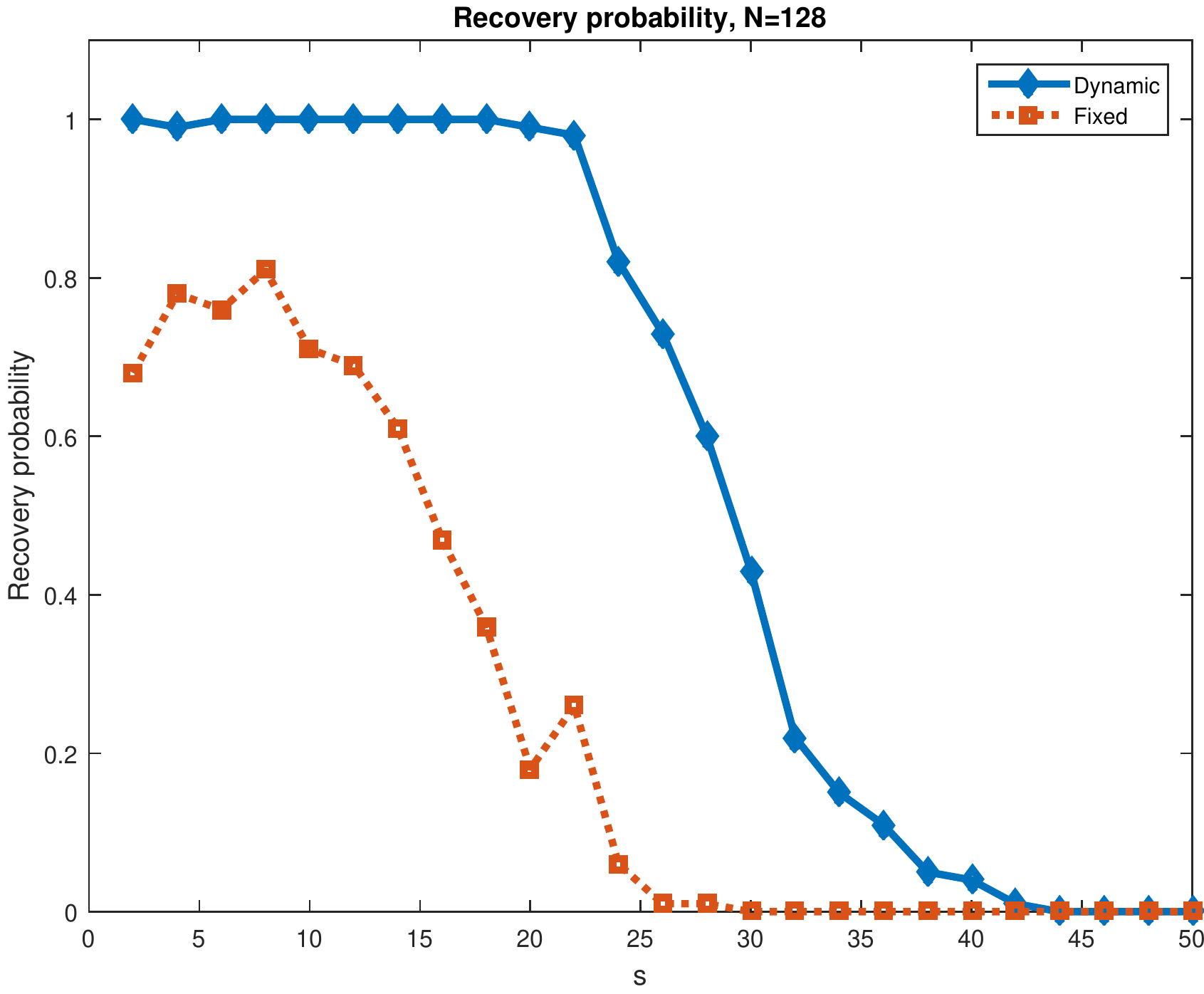}}
\caption{Performances comparison for dynamic steps versus fixed steps versus sparsity level.} \label{figstep}
\end{center}
\end{figure}

\section{Performances for coded diffraction pattern (CDP)}
In order to show that our proposed method can applied to a very general PR problem, numerical experiments for CDP are performed. The octanary CDP is explored, and specifically   each
element of $I_j$ in \eqref{cdp} takes a value randomly among the eight candidates, \emph{i.e.,} $\{\pm \sqrt{2}/2, \pm \sqrt{2}{\mathbf i}/2, \pm \sqrt{3},\pm \sqrt{3}{\mathbf i}\}$. Let $K=1,~2,~3,~4$, and $\lambda=2.0\times 10^{-2},~r_1=1.0\times 10^{-5},~r_2=1.0\times 10^{-6}$, $\tau=1.0005$. We use the L0L1PR as an example and plot the  recovery probability w.r.t. different sparse levels $s$ in Fig. \ref{cdpfig}. One can readily see that our proposed method can recovery the sparse signals with high probability. If oversampling by collecting multiple measurements by increasing $K$ , the proposed methods can handle more sophisticated signals.

\begin{figure}[!htb]
\begin{center}
\subfigure{\includegraphics[width=.24\textwidth]{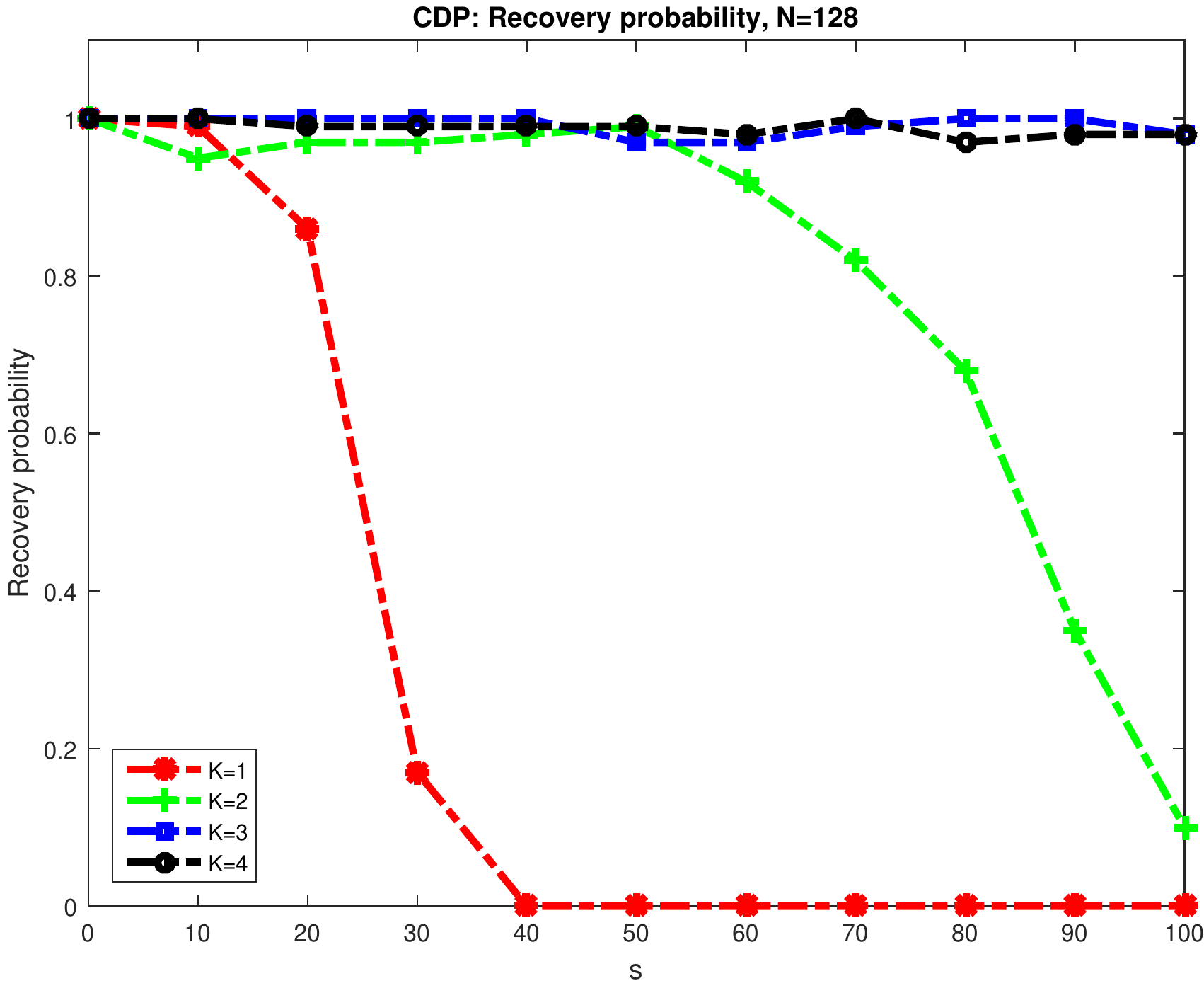}}
\caption{Recovery probability versus sparsity level  for CDP with $K=1,~2,~3,~4$.}
\label{cdpfig}
\end{center}
\end{figure}
\section{Conclusion}\label{sec6}
In this work, we study the general sparse PR problem based on variational approaches and  proposed novel variational PR models based on $L^0$  regularization and $L^p$ data fitting. By the fast implementation of ADMM to solve the proposed model, our methods can fast recover the sparse signals with  higher probability compared with the state of art algorithms. In the future, we will consider to generalize it to high dimension signals and further improve the convergence speed by domain decomposition methods \cite{chang2015convergence} and some multi-block coordinate splitting technique \cite{deng2013parallel}.

\section*{Acknowledgment}

The authors would like to thank Dr. Kishore Jaganathan to provide us the codes for SPR algorithm. Dr. Y. Duan was partially supported by the Ministry of Science and Technology of China (``863'' Program No.2015AA020101) and NSFC (NO.11526208).  Dr. Z.-F. Pang was partially supported by National Basic Research Program of China (973 Program No.2015CB856003) and NSFC (Nos.U1304610 and 11401170).   Dr. C. Wu was partially supported by NSFC (Nos.11301289 and 11531013). Dr. H. Chang was partially supported by China Scholarship Council (CSC) and NSFC (Nos.11426165 and 11501413).

\ifCLASSOPTIONcaptionsoff
  \newpage
\fi



%
\bibliographystyle{IEEEtran}

\begin{thebibliography}{10}
\providecommand{\url}[1]{#1}
\csname url@samestyle\endcsname
\providecommand{\newblock}{\relax}
\providecommand{\bibinfo}[2]{#2}
\providecommand{\BIBentrySTDinterwordspacing}{\spaceskip=0pt\relax}
\providecommand{\BIBentryALTinterwordstretchfactor}{4}
\providecommand{\BIBentryALTinterwordspacing}{\spaceskip=\fontdimen2\font plus
\BIBentryALTinterwordstretchfactor\fontdimen3\font minus
  \fontdimen4\font\relax}
\providecommand{\BIBforeignlanguage}[2]{{%
\expandafter\ifx\csname l@#1\endcsname\relax
\typeout{** WARNING: IEEEtran.bst: No hyphenation pattern has been}%
\typeout{** loaded for the language `#1'. Using the pattern for}%
\typeout{** the default language instead.}%
\else
\language=\csname l@#1\endcsname
\fi
#2}}
\providecommand{\BIBdecl}{\relax}
\BIBdecl

\bibitem{patterson1934fourier}
A.~L. Patterson, ``A fourier series method for the determination of the
  components of interatomic distances in crystals,'' \emph{Physical Review},
  vol.~46, no.~5, p. 372, 1934.

\bibitem{patterson1944ambiguities}
------, ``Ambiguities in the x-ray analysis of crystal structures,''
  \emph{Physical Review}, vol.~65, no. 5-6, p. 195, 1944.

\bibitem{walther1963question}
A.~Walther, ``The question of phase retrieval in optics,'' \emph{Journal of
  Modern Optics}, vol.~10, no.~1, pp. 41--49, 1963.

\bibitem{millane1990phase}
R.~P. Millane, ``Phase retrieval in crystallography and optics,'' \emph{JOSA
  A}, vol.~7, no.~3, pp. 394--411, 1990.

\bibitem{fienup1987phase}
C.~Fienup and J.~Dainty, ``Phase retrieval and image reconstruction for
  astronomy,'' \emph{Image Recovery: Theory and Application}, pp. 231--275,
  1987.

\bibitem{Sanz1985}
J.~L.~C. Sanz, ``Mathematical considerations for the problem of fourier
  transform phase retrieval from magnitude,'' \emph{SIAM J. Appl. Math.},
  vol.~45, pp. 651--664, 1985.

\bibitem{pardalos1991quadratic}
P.~M. Pardalos and S.~A. Vavasis, ``Quadratic programming with one negative
  eigenvalue is np-hard,'' \emph{Journal of Global Optimization}, vol.~1,
  no.~1, pp. 15--22, 1991.

\bibitem{gerchberg1972practical}
R.~W. Gerchberg, ``A practical algorithm for the determination of phase from
  image and diffraction plane pictures,'' \emph{Optik}, vol.~35, p. 237, 1972.

\bibitem{fienup1982phase}
J.~R. Fienup, ``Phase retrieval algorithms: a comparison,'' \emph{Applied
  optics}, vol.~21, no.~15, pp. 2758--2769, 1982.

\bibitem{elser2003phase}
V.~Elser, ``Phase retrieval by iterated projections,'' \emph{JOSA A}, vol.~20,
  no.~1, pp. 40--55, 2003.

\bibitem{marchesini2013augmented}
S.~Marchesini, A.~Schirotzek, C.~Yang, H.-t. Wu, and F.~Maia, ``Augmented
  projections for ptychographic imaging,'' \emph{Inverse Problems}, vol.~29,
  no.~11, p. 115009, 2013.

\bibitem{Luke2005}
D.~R. Luke, ``Relaxed averaged alternating reflections for diffraction
  imaging,'' \emph{Inverse Probl.}, vol.~21, no.~1, pp. 37--50, 2005.

\bibitem{Marchesini2007}
S.~Marchesini, ``Phase retrieval and saddle-point optimization,'' \emph{J. Opt.
  Soc. Am. A}, vol.~24, no.~10, pp. 3289--3296, 2007.

\bibitem{marchesini2007invited}
------, ``Invited article: A unified evaluation of iterative projection
  algorithms for phase retrieval,'' \emph{Review of scientific instruments},
  vol.~78, no.~1, p. 011301, 2007.

\bibitem{bauschke2002phase}
H.~H. Bauschke, P.~L. Combettes, and D.~R. Luke, ``Phase retrieval, error
  reduction algorithm, and fienup variants: a view from convex optimization,''
  \emph{JOSA A}, vol.~19, no.~7, pp. 1334--1345, 2002.

\bibitem{candes2015phaseWF}
E.~J. Candes, X.~Li, and M.~Soltanolkotabi, ``Phase retrieval via wirtinger
  flow: Theory and algorithms,'' \emph{Information Theory, IEEE Transactions
  on}, vol.~61, no.~4, pp. 1985--2007, 2015.

\bibitem{chen2015solving}
Y.~Chen and E.~Candes, ``Solving random quadratic systems of equations is
  nearly as easy as solving linear systems,'' in \emph{Advances in Neural
  Information Processing Systems}, 2015, pp. 739--747.

\bibitem{netrapalli2015phase}
P.~Netrapalli, P.~Jain, and S.~Sanghavi, ``Phase retrieval using alternating
  minimization,'' \emph{Signal Processing, IEEE Transactions on}, vol.~63,
  no.~18, pp. 4814--4826, 2015.

\bibitem{marchesini2015alternating}
S.~Marchesini, Y.-C. Tu, and H.-T. Wu, ``Alternating projection, ptychographic
  imaging and phase synchronization,'' \emph{Applied and Computational Harmonic
  Analysis}, 2015.

\bibitem{sun2016geometric}
J.~Sun, Q.~Qu, and J.~Wright, ``A geometric analysis of phase retrieval,''
  \emph{arXiv preprint arXiv:1602.06664}, 2016.

\bibitem{candes2015phaselift}
E.~J. Candes, Y.~C. Eldar, T.~Strohmer, and V.~Voroninski, ``Phase retrieval
  via matrix completion,'' \emph{SIAM Review}, vol.~57, no.~2, pp. 225--251,
  2015.

\bibitem{Waldspurger2012}
I.~Waldspurger, A.~Aspremont, and S.~Mallat, ``Phase recovery, maxcut and
  complex semidefinite programming,'' \emph{Math. Program., Ser. A}, pp. 1--35,
  2012.

\bibitem{goldstein2016phasemax}
T.~Goldstein and C.~Studer, ``Phasemax: Convex phase retrieval via basis
  pursuit,'' \emph{arXiv preprint arXiv:1610.07531}, 2016.

\bibitem{bahmani2016phase}
S.~Bahmani and J.~Romberg, ``Phase retrieval meets statistical learning theory:
  A flexible convex relaxation,'' \emph{arXiv preprint arXiv:1610.04210}, 2016.

\bibitem{donoho2006compressed}
D.~L. Donoho, ``Compressed sensing,'' \emph{Information Theory, IEEE
  Transactions on}, vol.~52, no.~4, pp. 1289--1306, 2006.

\bibitem{candes2006robust}
E.~J. Cand{\`e}s, J.~Romberg, and T.~Tao, ``Robust uncertainty principles:
  Exact signal reconstruction from highly incomplete frequency information,''
  \emph{Information Theory, IEEE Transactions on}, vol.~52, no.~2, pp.
  489--509, 2006.

\bibitem{marchesini:2003}
S.~Marchesini, H.~He, H.~N. Chapman, S.~P. Hau-Riege, A.~Noy, M.~R. Howells,
  U.~Weierstall, and J.~C. Spence, ``X-ray image reconstruction from a
  diffraction pattern alone,'' \emph{Physical Review B}, vol.~68, no.~14, p.
  140101, 2003.

\bibitem{moravec2007compressive}
M.~L. Moravec, J.~K. Romberg, and R.~G. Baraniuk, ``Compressive phase
  retrieval,'' in \emph{Optical Engineering+ Applications}.\hskip 1em plus
  0.5em minus 0.4em\relax International Society for Optics and Photonics, 2007,
  pp. 670\,120--670\,120.

\bibitem{ohlsson2012cprl}
H.~Ohlsson, A.~Yang, R.~Dong, and S.~Sastry, ``Cprl--an extension of
  compressive sensing to the phase retrieval problem,'' in \emph{Advances in
  Neural Information Processing Systems}, 2012, pp. 1367--1375.

\bibitem{li2013sparse}
X.~Li and V.~Voroninski, ``Sparse signal recovery from quadratic measurements
  via convex programming,'' \emph{SIAM Journal on Mathematical Analysis},
  vol.~45, no.~5, pp. 3019--3033, 2013.

\bibitem{Fienup1982}
J.~R. Fienup, ``Phase retrieval algorithms: a comparison,'' \emph{Appl. Opt.},
  vol.~21, no.~15, pp. 2758--2769, 1982.

\bibitem{mukherjee2014fienup}
S.~Mukherjee and C.~S. Seelamantula, ``Fienup algorithm with sparsity
  constraints: application to frequency-domain optical-coherence tomography,''
  \emph{IEEE Transactions on Signal Processing}, vol.~62, no.~18, pp.
  4659--4672, 2014.

\bibitem{yang2013robust}
Z.~Yang, C.~Zhang, and L.~Xie, ``Robust compressive phase retrieval via l1
  minimization with application to image reconstruction,'' \emph{arXiv preprint
  arXiv:1302.0081}, 2013.

\bibitem{shechtman2014gespar}
Y.~Shechtman, A.~Beck, and Y.~C. Eldar, ``Gespar: Efficient phase retrieval of
  sparse signals,'' \emph{Signal Processing, IEEE Transactions on}, vol.~62,
  no.~4, pp. 928--938, 2014.

\bibitem{schniter2015compressive}
P.~Schniter and S.~Rangan, ``Compressive phase retrieval via generalized
  approximate message passing,'' \emph{IEEE Transactions on Signal Processing},
  vol.~63, no.~4, pp. 1043--1055, 2015.

\bibitem{loock2014phase}
S.~Loock and G.~Plonka, ``Phase retrieval for fresnel measurements using a
  shearlet sparsity constraint,'' \emph{Inverse Problems}, vol.~30, no.~5, p.
  055005, 2014.

\bibitem{chang2016phase}
H.~Chang, Y.~Lou, M.~K. Ng, and T.~Zeng, ``Phase retrieval from incomplete
  magnitude information via total variation regularization,'' \emph{SIAM
  Journal on Scientific Computing}, vol.~38, no.~6, pp. A3672--A3695, 2016.

\bibitem{chang2016}
H.~Chang, Y.~Lou, and Y.~Duan, ``Total variation based phase retrieval for
  poisson noise removal,'' \emph{preprint}.

\bibitem{tillmann2016dolphin}
A.~M. Tillmann, Y.~C. Eldar, and J.~Mairal, ``Dolphin-dictionary learning for
  phase retrieval,'' \emph{arXiv preprint arXiv:1602.02263}, 2016.

\bibitem{qiu2016undersampled}
T.~Qiu and D.~P. Palomar, ``Undersampled phase retrieval via
  majorization-minimization,'' \emph{arXiv preprint arXiv:1609.02842}, 2016.

\bibitem{chang2016general}
H.~Chang and S.~Marchesini, ``A general framework for denoising phaseless
  diffraction measurements,'' \emph{arXiv preprint arXiv:1611.01417}, 2016.

\bibitem{ohlsson2014conditions}
H.~Ohlsson and Y.~C. Eldar, ``On conditions for uniqueness in sparse phase
  retrieval,'' in \emph{2014 IEEE International Conference on Acoustics, Speech
  and Signal Processing (ICASSP)}.\hskip 1em plus 0.5em minus 0.4em\relax IEEE,
  2014, pp. 1841--1845.

\bibitem{eldar2014phase}
Y.~C. Eldar and S.~Mendelson, ``Phase retrieval: Stability and recovery
  guarantees,'' \emph{Applied and Computational Harmonic Analysis}, vol.~36,
  no.~3, pp. 473--494, 2014.

\bibitem{wang2014phase}
Y.~Wang and Z.~Xu, ``Phase retrieval for sparse signals,'' \emph{Applied and
  Computational Harmonic Analysis}, vol.~37, no.~3, pp. 531--544, 2014.

\bibitem{Shechtman2014}
Y.~Shechtman, Y.~C. Eldar, O.~Cohen, H.~N. Chapman, J.~Miao, and M.~Segev,
  ``Phase retrieval with application to optical imaging: a contemporary
  overview,'' \emph{Signal Processing Magazine, IEEE}, vol.~32, no.~3, pp.
  87--109, 2015.

\bibitem{Wu&Tai2010}
C.~Wu and X.-C. Tai, ``Augmented {Lagrangian} method, dual methods and
  split-{Bregman} iterations for {ROF}, vectorial {TV} and higher order
  models,'' \emph{SIAM J. Imaging Sci.}, vol.~3, no.~3, pp. 300--339, 2010.

\bibitem{papadimitriou1998combinatorial}
C.~H. Papadimitriou and K.~Steiglitz, ``Combinatorial optimization: Algorithms
  and complexity,'' 1998.

\bibitem{candes2015phaseCDP}
E.~J. Candes, X.~Li, and M.~Soltanolkotabi, ``Phase retrieval from coded
  diffraction patterns,'' \emph{Applied and Computational Harmonic Analysis},
  vol.~39, no.~2, pp. 277--299, 2015.

\bibitem{duan2015regularized}
Y.~Duan, H.~Chang, W.~Huang, J.~Zhou, Z.~Lu, and C.~Wu, ``The regularized
  mumford--shah model for bias correction and segmentation of medical images,''
  \emph{IEEE Transactions on Image Processing}, vol.~24, no.~11, pp.
  3927--3938, 2015.

\bibitem{wu2011augmented}
C.~Wu, J.~Zhang, and X.-C. Tai, ``Augmented lagrangian method for total
  variation restoration with non-quadratic fidelity,'' \emph{Inverse problems
  and imaging}, vol.~5, no.~1, pp. 237--261, 2011.

\bibitem{bolte2014proximal}
J.~Bolte, S.~Sabach, and M.~Teboulle, ``Proximal alternating linearized
  minimization for nonconvex and nonsmooth problems,'' \emph{Mathematical
  Programming}, vol. 146, no. 1-2, pp. 459--494, 2014.

\bibitem{wang2015global}
Y.~Wang, W.~Yin, and J.~Zeng, ``Global convergence of admm in nonconvex
  nonsmooth optimization,'' \emph{arXiv preprint arXiv:1511.06324}, 2015.

\bibitem{hong2016convergence}
M.~Hong, Z.-Q. Luo, and M.~Razaviyayn, ``Convergence analysis of alternating
  direction method of multipliers for a family of nonconvex problems,''
  \emph{SIAM Journal on Optimization}, vol.~26, no.~1, pp. 337--364, 2016.

\bibitem{wen2012alternating}
Z.~Wen, C.~Yang, X.~Liu, and S.~Marchesini, ``Alternating direction methods for
  classical and ptychographic phase retrieval,'' \emph{Inverse Problems},
  vol.~28, no.~11, p. 115010, 2012.

\bibitem{chang2015convergence}
H.~Chang, X.-C. Tai, L.-L. Wang, and D.~Yang, ``Convergence rate of overlapping
  domain decomposition methods for the rudin--osher--fatemi model based on a
  dual formulation,'' \emph{SIAM Journal on Imaging Sciences}, vol.~8, no.~1,
  pp. 564--591, 2015.

\bibitem{deng2013parallel}
W.~Deng, M.-J. Lai, Z.~Peng, and W.~Yin, ``Parallel multi-block admm with o
  (1/k) convergence,'' \emph{arXiv preprint arXiv:1312.3040}, 2013.

\end{thebibliography}

%
%
%

%




\end{document}